\documentclass[12pt]{amsart}
\usepackage{amssymb, amsmath, amsthm}
\usepackage[utf8]{inputenc}
\usepackage{graphicx}
\usepackage{euscript}
\usepackage{dsfont}
\usepackage{color}
\usepackage{xcolor}
\usepackage{bm}
\usepackage{tikz}
\usetikzlibrary{shapes,shapes.geometric,arrows,fit,calc,positioning,arrows,automata,snakes,}
\definecolor{myred}{rgb}{0.2,0,0}
\definecolor{myblue}{rgb}{0,0,0.6}
\definecolor{mygreen}{rgb}{0,0.2,0}
\usepackage[hyphens]{url}
\usepackage[colorlinks=true,linkcolor=myred,citecolor=myblue,urlcolor=mygreen]{hyperref}
\usepackage[centering]{geometry}
\usepackage[english]{babel}
\usepackage{bbm}
\usepackage{textcomp}

\usepackage[numbers,sort&compress]{natbib}
\usepackage[sectionbib]{bibunits}
\defaultbibliographystyle{plainnat}

\newcommand{\R}{\mathbb{R}}
\newcommand{\C}{\mathbb{C}}
\newcommand{\N}{\mathbb{N}}
\newcommand{\Z}{\mathbb{Z}}

\newcommand{\T}{\mathbb{T}}

\newcommand{\norm}[1]{\left\lVert#1\right\rVert}

\newcommand{\les}{\leqslant}
\newcommand{\ges}{\geqslant}
\newcommand{\floor}[1]{\left\lfloor #1 \right\rfloor}

\newtheorem*{conjecture*}{Conjecture}
\newtheorem{theorem}{Theorem}[section]
\newtheorem*{theorem*}{Theorem}
\newtheorem{lemma}{Lemma}[section]

\newtheorem{proposition}{Proposition}[section]

\newtheorem{remark}{Remark}[section]

\tikzset{
	state/.style={draw,ellipse}
}

\numberwithin{equation}{section}

\begin{document}
\title[M\"{o}bius orthogonality of Thue-Morse sequence]{M\"{o}bius orthogonality of Thue-Morse sequence along Piatetski-Shapiro numbers} 
\author{Andrei Shubin}
\address{Institute of Discrete Mathematics and Geometry, TU Wien, Wiedner Hauptstr. 8-10, A-1040 Wien, Austria}
\email{\href{mailto:andrei.shubin@tuwien.ac.at}{andrei.shubin@tuwien.ac.at}}

\maketitle

\begin{abstract}
	We show that the M\"{o}bius function is orthogonal to the Thue-Morse sequence $t(n)$ taken along the Piatetski-Shapiro numbers $\floor{n^c}$ for any $1 < c < 2$. Previously this property was established for the subsequence along the squares $t(n^2)$. These are both examples of M\"{o}bius orthogonal sequences with maximum entropy. 
\end{abstract}

\section{Introduction}
\label{Intro}

The Thue-Morse sequence can be defined as the sum of digits modulo 2 in the binary representation of $n$. For example, 
$$
	t(1) = 1, \qquad t(2) = t(10_2) = 1, \qquad t(3) = t(11_2) = 0, \qquad t(4) = t(100_2) = 1,
$$ and so forth. This sequence can also be viewed as an output of a \textit{determenistic finite automation} with two states, and the input and output alphabets consisting of letters $0$ and $1$:

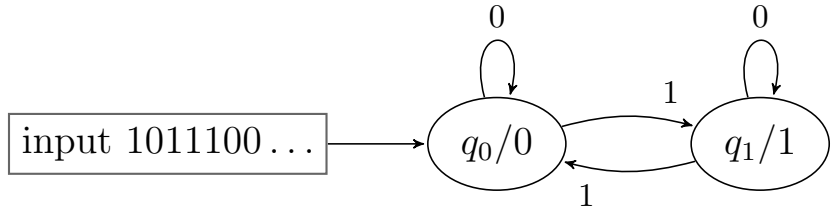
\begin{center}
\begin{tikzpicture}[->,>=stealth',shorten >=1pt,auto,node distance=2.8cm, semithick, bend angle = 15, cross line/.style={preaction={draw=white,-,line width=4pt}}, scale=1.25, transform shape, squarednode/.style={rectangle, draw=black!60, thick, minimum size=5mm}]

\node[squarednode](I)		{input $1011100\ldots$};
\begin{scope}[node distance=3.5cm]
\node[state](A)     [right of=I]              {$q_0 / 0$};
\end{scope}
\node[state]         (B) [right of=A] 			{$q_1 / 1$};

\path [every node/.style={font=\footnotesize}, pos = 0.66]
(I) edge 		  node 		 {}      (A)
(A) edge 	[bend left]	  node 		 {1}      (B)
edge [loop above] node [pos=0.5] {0} 	(A)
(B) edge 	[bend left]	  node 		 {1}      (A)
edge [loop above] node 		[pos=0.5] 	{0}	(B);
\end{tikzpicture}
\end{center}


Therefore the Thue-Morse sequence is an example of a \textit{2-automatic sequence}. A detailed theory of automatic sequences can be found in the book of Allouche and Shallit~\cite{All_Shall_book}. For an overview of the results about the Thue-Morse sequence see also~\cite{All_Shall_TM}.

An analogue of a prime number theorem (PNT) for the Thue-Morse sequence was established in the work by Maudit and Rivat~\cite{MR2010}:
$$
	\sum_{n \les N} \Lambda(n) t(n) = \frac{N}{2} + o( N ) \qquad N \to +\infty.
$$ Their result is in fact more general as it implies the equidistribution of sums of digits of primes in arbitrary base among arithmetic progressions. In particular, it gave a solution for one of three open problems stated by Gelfond in~\cite{Gelfond68}. The M\"{o}bius orthogonality for the Thue-Morse sequence, namely the estimate
\begin{equation} \label{Sarnak}
	\sum_{n \les N} \mu(n) t(n) = o(N),
\end{equation} which is a weaker statement, was already known from earlier works of Indlekofer and Katai~\cite{Indle_Kat, Katai86}, and Dartyge and Tenenbaum~\cite{Dart_Tanen}, but it also follows from the work of Maudit and Rivat.

The M\"{o}bius randomness principle states that any ``reasonable'' bouned function $f(n)$ in place of $t(n)$ should satisfy a relation similar to~\eqref{Sarnak}. This principle is described by the famous \textit{Sarnak conjecture}. More precisely, it says that for any continuous $f : X \to \R$ and $x \in X$, one has
$$
	\sum_{n \les N} \mu(n) f(T^n x) = o(N),
$$ where  $(X, T)$ is a dynamical system with zero topological entropy.

The Sarnak Conjecture for all automatic sequences was verified by M\"ullner in~\cite{Mullner}, where he also proved a PNT for sequences generated by strongly connected automata. For other instances, when the Sarnak Conjecture and/or a PNT was verified see~\cite{Ferenczi, Kanig_Lem_Radz, DMS18, DMS_126, Kul_Leman}.

In the case of automatic sequences the topological entropy is related to the \textit{subword complexity} of the sequence. For a given sequence $u(n)$ and a fixed number $H > 0$ the subword complexity is the number of different patterns of length $H$:
$$
	\# \bigl\{ \bigl( u(n), u(n+1), \ldots, u(n + H -1) \bigr) \ \ \forall n \in \N \bigr\}.
$$ It is known that the subword complexity of any automatic sequence is bounded by a linear function of $H$. In particular, for the Thue-Morse sequence the upper limit is known to be $(10/3) H$ (see~\cite{Srecko, de_Luca_Var, Avg}). In contrast, the subsequence $t(n^2)$ is \textit{normal}, which means that each of $2^H$ patterns occurs as a subword with an asymptotic density $2^{-H}$ (see~\cite{DMR-squares}). The same property holds for Piatetski-Shapiro subsequences $t(\floor{n^c})$ when $1 < c < 3/2$~\cite{Mull-Spieg}. This result might be extended to all exponents $1 < c < 2$ with the existing techniques.

Although the sequences $t(n^2)$ and $t(\floor{n^c})$ are much more ``randomized'' compared to the original Thue-Morse sequence, one can still show M\"{o}bius orthogonality for them, albeit with slightly different approaches. For squares this was established in~\cite{Sarnak_squares}. In this work we cover the Piatetski-Shapiro case:
\begin{theorem} \label{thm1}
	Let $c$ be a fixed real number such that $1 < c < 2$ and $t(n)$ the Thue-Morse sequence. Then
	$$
		\sum_{n \les N} \mu(n) t(\floor{n^c}) = o(N) \qquad \text{when } N \to +\infty.
	$$
\end{theorem}

Since $t(\floor{n^c})$ for $1 < c < 3/2$ is a normal sequence it has maximum entropy. Hence, Theorem~\ref{thm1} provides further examples for M\"{o}bius orthogonal sequences having this property, see also~\cite{Sarnak_squares}. 

We should also mention that there are automatic sequences whose behavior is significantly different from the Thue-Morse one, namely \textit{synchronizing} automatic sequences $a(n)$. They are generated by an automaton with a \textit{synchronizing word}, which means that the automaton returns to a given state every time a given synchronizing subword appears in the input.
In fact, this is a large class of sequences as it is known from the work of Berlinkov~\cite{Berlinkov} that almost all automata are synchronizing. Synchronizing automatic sequences are well approximated by periodic ones. This allows in particular to get a PNT for $a(\floor{n^c})$ by exploring the distribution of $\floor{p^c}$ in arithmetic progressions, which can be done by a classical approach of Vinogradov and Vaughan. On the other hand, the subword complexity of $a(\floor{n^c})$ is subexponential, which means they are much less ``randomized'' compared to $t(\floor{n^c})$. The questions on PNT and subword complexity for synchronizing automatic sequences will be discussed in detail in our forthcoming work~\cite{DDMSS_future}.

\subsection{Sketch of the proof}

The proof of Theorem~\ref{thm1} is based on the techniques developed recently by Spiegelhofer for the level of distribution of \textit{Beatty} subsequences $t(\floor{\alpha n + \beta})$~\cite{Spieg} and by Drmota-M\"{u}llner-Spiegelhofer for primes along the sums of Fibonacci numbers~\cite{DMS_126}.

The Vinogradov-Vaughan approach for a PNT does not seem to work in this case as it requires a non-trivial estimate for a type II sum with a factor of the form $t(\floor{(mn)^c})$. In the most critical range,
$$
	\sum_{m \sim \sqrt{N}} \sum_{n \sim \sqrt{N}} v_m w_n t(\floor{mn}^c),
$$ this is apparently beyond the limitation of the approaches from~\cite{Spieg} and~\cite{DMS_126}.

M\"{o}bius orthogonality is an easier problem as there are more tools for decomposing the original sum with $\mu(n)$ compared to the case of $\Lambda(n)$. This relies on the fact that most numbers are composite and have small prime factors. To get rid of the factor $\mu(n)$ we apply a variation of a theorem of Daboussi, also known as the Daboussi-Delange-Katai-Bourgain-Sarnak-Ziegler (DDKBSZ) criterion (see~\cite{Daboussi, Dab-Del, Katai86, BSZ13}. It reduces the problem to showing that for a typical pair of primes $p$ and $q$ not exceeding $N^{\theta}$ for some small $\theta > 0$ one has  
$$
	\sum_{n \les N} t(\floor{(pn)^c}) t(\floor{(qn)^c}) = o(N).
$$ This approach allows one to avoid estimating sums with additional large factors inside the argument. We will need the following averaged version of the Daboussi's theorem:
\begin{proposition}[modification of DDKBSZ] \label{prop1}
	Let $f: \N \to \C$ be a bounded function, $\theta > 0$ be a fixed small number, $N > 0$ be a large number, $\mathbb{P}$ be a set of primes, and $\mathbb{P}_k = \mathbb{P}_k (N, \theta) = \mathbb{P} \cap (2^k, 2^{k+1}] \cap [1, N^{\theta}]$. Further, assume that the relation
	\begin{equation} \label{Daboussi_eq}
		\sum_{p, q \in \mathbb{P}_k} \sum_{n \les \min(N/p, N/q)} f(pn) \overline{f(qn)} = o\left( \frac{N |\mathbb{P}_k|^2 }{2^k} \right)
	\end{equation} holds uniformly for all $k$ with $N^{\theta/2} \les 2^k \les N^{\theta}$. Then
	$$
		\sum_{n \les N} \mu(n) f(n) = o(N).
	$$
\end{proposition}

Proposition~\ref{prop1} follows directly from the argument of Tao~\cite{Tao2011}. When $f(n) = t(\floor{n^c})$ we will show that for most pairs $p, q \in \mathbb{P}_k (N, \theta)$, where $\theta$ will depend on $c$ in Theorem~\ref{thm1}, the inner sum in~\eqref{Daboussi_eq} is $o(N 2^{-k})$, whereas for the remaining few pairs we will apply a trivial bound $O(N 2^{-k})$.

In the first case the idea is to show that for all ``good'' pairs $(p, q)$ the quantities
\begin{equation} \label{quantities}
	\left| \frac{1}{N} \# \bigl\{ n \in (N, 2N]: t(\floor{(pn)^c}) = \alpha_1, \ t(\floor{(qn)^c}) = \alpha_2 \bigr\} - \frac{1}{4} \right|
\end{equation} are $o(1)$ for all four patterns $(\alpha_1, \alpha_2) \in \{ (0, 0), (0, 1), (1, 0), (1, 1) \}$. The reason for this is that for most pairs the number $(q/p)^c$ cannot be well approximated by rationals, and so the functions $t(\floor{(pn)^c})$ and $t(\floor{(qn)^c})$ behave ``independently'' from each other. In the latter case, one can develop a 2-dimensional analogue of Spiegelhofer's approach.

The quantities~\eqref{quantities} are closely related to a problem on the level of distribution of a Beatty subsequence $t(\floor{\alpha n + \beta})$ as $\alpha$ varies around $N^{\rho_1}$ for an appropriate exponent $\rho_1 = \rho_1 (c) > 0$. Without loss of generality we may only consider the pattern $(1, 1)$. The necessary bound is given by the following proposition:
\begin{proposition} \label{prop2}
	Let $N \ges K \ges 2$ be large numbers, $\theta > 0$, $k$, $\mathbb{P}_k (N, \theta)$ be as in Proposition~\ref{prop1}, $c > 1$ be as in Theorem~\ref{thm1}, and $p, q$ be distinct primes, such that $p, q \in P_k (N, \theta)$. Then we have
	\begin{multline} \label{prop2_ineq}
		\left| \frac{1}{N} \# \bigl\{ n \in (N, 2N]: t(\floor{(pn)^c}) = 1, \ t(\floor{(qn)^c}) = 1 \bigr\} - \frac{1}{4} \right| \ll \\
		2^{kc} N^{c-2} K^2 + \frac{(\log N)^2}{K} + J(N, K),
	\end{multline} where
	\begin{multline*}
		J(N, K) := \frac{1}{2^{kc} N^{c-1}} \int_{c 2^{kc} N^{c-1}}^{c 2^{kc} (2N)^{c-1}} \max_{\beta_1, \beta_2 \ges 0} \biggl| \frac{1}{K} \# \bigl\{ n < K:  t(\floor{\alpha n + \beta_1}) = 1, \\ 
		t(\floor{\gamma \alpha n + \beta_2}) = 1 \bigr\} - \frac{1}{4} \biggr| d\alpha,
	\end{multline*} and $\gamma := (q/p)^c$. 
\end{proposition}

\begin{remark}
	Note that since $p$ and $q$ lie in the same dyadic interval, the number $\gamma$ is of size $1$.
\end{remark}

In order to apply Proposition~\ref{prop1}, we need to show that for most pairs $(p, q)$ the right side of~\eqref{prop2_ineq} is $o(1)$. Thus, we need to have $K \gg (\log N)^{2+\varepsilon}$, $K \ll N^{1 - c/2 - c\theta - \varepsilon}$, and $J(N, K) \ll o(1)$. That means $\theta$ has to be chosen so that the inequality $1 - c/2 - c\theta > 0$ is satisfied. Note that $\varepsilon > 0$ can be arbitrarily small with respect to $\theta$. Moreover, the condition $J(N, K) \ll o(1)$ when $c \to 2-$ corresponds to an arbitrarily large $\alpha$ with respect to $n$. That means that a 2-dimensional Beatty subsequence ( $t(\floor{\alpha n + \beta_1})$, $t(\floor{\gamma \alpha n + \beta_2})$) should have a level of distribution 1. The latter problem can be reduced to estimating an exponential sum with a sum of digits function. We denote by $s(n)$ the sum of digits of $n$ in its binary representation. Then for any integers $n_1, n_2 > 0$ one clearly has
$$
	t(n_1) t(n_2) = \frac{1}{2} \left( 1 - e\bigl( \frac{1}{2} s(n_1) \bigr) \right) \cdot \frac{1}{2} \left( 1 - e\bigl( \frac{1}{2} s(n_2) \bigr) \right).
$$ The desired bound for $J(N, K)$ would follow from the corresponding estimates for the quantities
\begin{gather*}
	\int_{c 2^{kc} N^{c-1}}^{c 2^{kc} (2N)^{c-1}} \max_{\beta_1 > 0} \biggl| \sum_{n \les K} e \left( \frac{1}{2} s(\floor{\alpha n + \beta_1}) \right) \biggr| d\alpha, \\
	\int_{c 2^{kc} N^{c-1}}^{c 2^{kc} (2N)^{c-1}} \max_{\beta_1, \beta_2 > 0} \biggl| \sum_{n \les K} e \left(\frac{1}{2} s(\floor{\alpha n + \beta_1}) + \frac{1}{2} s \left( \floor{\gamma \alpha n + \beta_2}\right) \right) \biggr| d\alpha.
\end{gather*} The necessary bound for the first expression follows from~\cite{Spieg}. For the second expression we prove the following estimate, which is a 2-dimenstional analogue of Proposition~3.2 from~\cite{Spieg}:

\begin{proposition} \label{prop3}
	Assume that $D, N > 0$ are large real numbers such that $N^{\rho_1} \les D \les N^{\rho_2}$ for some fixed $\rho_2 > \rho_1 \ges 10$, where $\rho_2, \rho_1$ depend on $c$ in Theorem~\ref{thm1}. Furthermore, let $\gamma$ be as in Proposition~\ref{prop2} satisfying also the following restriction: there is $\theta_1 > 0$ such that for any rational number $h_1 / h_2$ with $|h_2| \les N^{\theta_1}$ one has
	\begin{equation} \label{restrict_irrat}
		\left| \gamma - \frac{h_1}{h_2} \right| > \frac{1}{N^{1/2}}.
	\end{equation} Then there exist $\eta > 0, c_0 > 0$, such that
	\begin{equation} \label{expsum}
		\int_D^{2D} \max_{\beta_1, \beta_2 \ges 0} \biggl| \sum_{n \les N} e\biggl( \frac{1}{2} s(\floor{\alpha n + \beta_1}) + \frac{1}{2} s \bigl( \floor{\gamma \alpha n + \beta_2}\bigr) \biggr) \biggr| \les c_0 DN^{1-\eta}.
	\end{equation} The numbers $\theta_1, \eta, c_0$ are independent of $N$ and $D$.
\end{proposition}

There are essenatially three steps in the proof of~\eqref{expsum}. 

The first idea concerns the fact that the difference $s(\floor{\alpha (n + k) + \beta}) - s(\floor{\alpha n + \beta})$ does not depend on the leading digits of $\floor{\alpha n + \beta}$ \textit{for most} values of $\alpha$ if $k$ is small compared to $n$. The shifted sums can be obtained in the standard way by a van der Corput type inequalities. Since the leading digits mostly do not affect the difference, one can replace the original function by its truncated version $s_{\lambda} (n)$, which is a sum of the last $\lambda$ digits and thus $2^{\lambda}$-periodic:
$$
	s(\floor{\alpha (n + k) + \beta}) - s(\floor{\alpha n + \beta}) = s_{\lambda}(\floor{\alpha (n + k) + \beta}) - s_{\lambda}(\floor{\alpha n + \beta}) \quad \text{for most } \alpha.
$$ This is usually referred to as \textit{carry property}. The precise version of this property is given by Lemma~\ref{propagation} (\textit{carry propagation lemma}). It was introduced in the works of Maudit and Rivat~\cite{MR2009, MR2010}, and later used in~\cite{Mull-Spieg, Spieg, DMS_126}. The important advantage of the truncated function is its periodicity. It allows one to explore the distribution of $s_{\lambda} (n)$ by looking at the fractional part of $n / 2^{\lambda}$, which can be handled by a Fourier analysis. A somewhat similar idea is used in~\cite{DMS_126} for the sums of Fibonacci numbers where one needs to study the fractional part of $\gamma n$, where $\gamma$ is the golden ratio. Going back to the Thue-Morse case, the distribution of $\{ n / 2^{\lambda} \}$ is easier to handle when $\lambda$ is small, as the corresponding subintervals of $[0, 1)$ are large. However, the error terms produced by carry propagation lemma force a large choice of $\lambda$. We truncate the funtion $s_{\lambda}$ a few more times using the techniques for ``cutting away'' of the digits from~\cite{Spieg, DMS_126}. There are limitations on how many digits one can cut at each step, and this is the reason why we apply the generalized van der Corput inequality several times. Finally, we will end up having the function $s_{\rho} (n)$, where $2^{\rho}$ is a small power of $N$.

For the second step we essentially replace the multiple sum of the form
\begin{multline*}
	\frac{1}{N} \sum_{n \les N} \frac{1}{|K_1| \ldots |K_m|} \sum_{\substack{k_1, \ldots, k_m \\ k_i \in K_i}} e\biggl( \frac{1}{2} \sum_{\varepsilon_1, \ldots, \varepsilon_m \in \{0, 1\}} \biggl( g_{\rho} \left( \frac{\alpha n}{2^{\rho}} + \frac{\alpha}{2^{\rho}} \langle \bm{\varepsilon}, \bm{k} \rangle \right) + \\
	g_{\rho} \left( \frac{\gamma \alpha n}{2^{\rho}} + \frac{\gamma \alpha}{2^{\rho}} \langle \bm{\varepsilon}, \bm{k} \rangle \right) \biggr) \biggr) 
\end{multline*} by the multiple integral of the form
\begin{equation} \label{integral_norm}
	\int_{[0, 1]} \int_{[0, 1]^m} e \biggl( \frac{1}{2} \sum_{\varepsilon_1, \ldots, \varepsilon_m \in \{0, 1\}} g_{\rho} \bigl( x + \langle \bm{\varepsilon}, \bm{x} \rangle \bigr) \biggr) dx dx_1 \ldots dx_m =: I_m (g_{\rho})
\end{equation} using Koksma type inequalities. Here $g_{\rho}(.)$ is a certain 1-periodic piecewise constant function obtained from the truncated sum of digits $s_{\rho} (.)$, $\langle \bm{\varepsilon}, \bm{x} \rangle$ denotes an inner product $\varepsilon_1 x_1 + \ldots + \varepsilon_m x_m$, and $K_i = K_i (\alpha)$ are certain subsets of $\{ 1, \ldots, N \}$ with the restrictions on the digits of $\floor{\alpha k}$ and $\floor{\gamma \alpha k}$. Recall that in our case $\gamma = (q/p)^c$. The sums over $k_i$ will appear after the application of generalized van der Corput inequality (see Lemma~\ref{generalized_Corput}). Moving to integrals is based on discrepancy estimates for $\{ \alpha n / 2^{\rho} \}$ and $\{ \gamma \alpha n / 2^{\rho} \}$, which are small for most $\alpha$. However, since we deal with both of these sequences at the same time, we need a discrepancy estimate for the joint distribution of $\{ \alpha n / 2^{\rho} \}$ and $\{ \gamma \alpha n / 2^{\rho}  \}$. This is small enough only when $\gamma$ is far enough from any rational number with a small denominator. In this case these sequences behave roughly ``independently'' from each other. The need for this independence is the reason why we exclude certain pairs of primes in Proposition~\ref{prop1}.

In the last step of the proof we use a bound for a Gowers $m$-uniformity norm of $g_{\rho}(.)$, which is known to decay exponentially fast on the number of digits. From the previous steps we find that the expression in~\eqref{expsum} is bounded roughly by $DN I_m (g_{\rho})^{-C_m}$ with some $C_m > 0$. The integral $I_m (g_{\rho})$ can be seen as the continuous version of the Gowers norm of the function $e(g_{\rho} (x)/2)$ on the torus $\T = \R / \Z$. For an arbitrary complex-valued function $f$ it can be defined as
$$
	\norm{f}_{U^m (\T)}^{2^m} := \int_{[0,1]} \int_{[0,1]^m} \prod_{\varepsilon_1, \ldots, \varepsilon_m \in \{0, 1\}} \mathcal{C}^{\varepsilon_1 + \ldots + \varepsilon_m} f\bigl( x + \langle \bm{\varepsilon}, \bm{x} \rangle \bigr) dx dx_1 \ldots dx_m,
$$ where $\mathcal{C}$ is the conjugation operator. For a more detailed theory of Gowers norms see, for example,~\cite{Gow01, Green2007, Tao2012, Host_Kra_2005, Host_Kra_2012, Kon19, BKM20}. An upper bound for the discrete Gowers norm of the truncated Thue-Morse sequence was obtained by Konieczny in~\cite{Kon19}. In his version the function is assumed to be zero for $n > 2^{\rho}$. For the periodic sum of digits function $s_{\rho}(n)$ the argument of Konieczny was adjusted by Spiegelhofer (see~\cite[Proposition~3.3]{Spieg}). He obtained an estimate of the form
$$
	\frac{1}{2^{(m+1)\rho}} \sum_{\substack{0 \les n < 2^{\rho} \\ 0 \les r_1, \ldots, r_m < 2^{\rho}}} e\biggl( \frac{1}{2} \sum_{\varepsilon_1, \ldots, \varepsilon_m \in \{0, 1\}} s_{\rho} (n + \langle \bm{\varepsilon}, \bm{r} \rangle ) \biggr) \ll 2^{-\eta \rho}
$$ with some $\eta > 0$. In Section~\ref{pf_Gowers_norm} we show that the integral norm~\eqref{integral_norm} of $g_{\rho}$ is bounded by the discrete norm of $s_{\rho}$ using a Gowers modification of Cauchy inequality~\cite[Lemma~3.8]{Gow01}. This is given by Lemma~\ref{Gowers_norm_estimate}. We then obtain
$$
	\norm{e(g_{\rho})}_{U^m (\T)}^{2^m} \ll 2^{-\tilde \eta \rho} \qquad \text{for some } \tilde \eta = \tilde \eta (m) > 0,
$$ which implies Proposition~\ref{prop3} as soon as $2^{\tilde \eta \rho}$ is large compared to any additional factors arising after the second step.

\subsection{Plan of the paper}

In Section~\ref{pf_thm1} we deduce Theorem~\ref{thm1} from Propositions~\ref{prop1}, \ref{prop2}, and \ref{prop3} by showing that the number of ``bad'' pairs of primes $(p, q)$ is small. In Section~\ref{pf_prop2} we prove Proposition~\ref{prop2}. The rest of the paper is devoted to the proof of Proposition~\ref{prop3}.

\subsection{Notation}

We use the standard notation for the real character $e(x) := e^{2 \pi i x}$, the fractional part $\{ x \} := x - \floor{x}$, and the distance to the nearest integer $\norm{x} = \min_{n \in \Z} |x - n|$. The notation $\norm{f}_{U^m}$ is used for the Gowers $m$-norm.

The relations  $f(x) \ll g(x)$ or $f(x) = O(g(x))$ mean $|f(x)| \les C g(x)$ for some fixed number $C > 0$ for all large enough $x$. The relation $f(x) \asymp g(x)$ means $f(x) \ll g(x)$ and $g(x) \ll f(x)$ at the same time. Finally, $f(x) = o(g(x))$ means $f(x) / g(x) \to 0$ as $x \to +\infty$.

We denote by $s(n)$ the sum of binary digits of $n$, and by $s_{\lambda} (n)$ the sum of the last $\lambda$ digits. The digit of $n$ in the $j$-th position from the right is denoted as $\delta_j (n)$.

The inner product $x_1 y_1 + \ldots + x_m y_m$ is denoted as $\langle \bm{x}, \bm{y} \rangle$. The dimension $m$ varies throughout the paper.

Finally, we use the standard notation for the discrepancy of the sequence $\bm{x} = \{ \bm{x}_1, \bm{x}_2, \bm{x}_3, \ldots \}$ on the unit cube $[0, 1]^k$:
$$
	D_N (\bm{x}) := \sup_{J} \left| \frac{1}{N} \# \bigl\{ i \les N : \bm{x}_i \bmod{1} \in J \bigr\} - \lambda(J) \right|,
$$ where $J$ runs through all subintervals of $[0, 1]^k$, and $\lambda(J)$ denotes its Lebesgue measure. Furthermore, we denote by $D_{N,\mathcal{A}} (\bm{x})$ the discrepancy of the subsequence of $\bm{x}$ along $\mathcal{A} \subset \N$.

\subsection{Acknowledgements} 

The work is supported by the Austrian-French project ``Arithmetic Randomness'' and the Austrian Science Fund FWF (I 4945-N). The author thanks Michael Drmota, Clemens M\"{u}llner, and Lukas Spiegelhofer for introducing him to this topic and many helpful conversations.

\section{Proof of Theorem~\ref{thm1}}
\label{pf_thm1}

To prove Theorem~\ref{thm1} it is enough to verify the assumptions of Proposition~\ref{prop1} when $f(n) = t(\floor{n^c})$. We show that most pairs of primes satisfy the restriction~\eqref{restrict_irrat} by evaluating from above the number of exceptional pairs. Then the contribution from the ``bad'' pairs can be estimated trivially as
$$
	\sum_{n \les \min(N/p, N/q)} t(\floor{(pn)^c}) t(\floor{qn}^c) \les \frac{N}{2^k},
$$ whereas the contribution from the ``good'' pairs by Propositions~\ref{prop2} and~\ref{prop3} is $o(N 2^{-k})$ uniformly in $N^{\theta/2} \les p, q \les 2 N^{\theta}$. Thus, it is enough to show that the exceptional set has an asymptotic density of zero.

We can crudely bound the number of ``bad'' pairs of primes by the number of ``bad'' pairs of integers $n, m \in (2^k, 2^{k+1}]$,
$$
	\left| \left( \frac{n}{m} \right)^c - \frac{a}{b} \right| < \frac{1}{N^{1/2 + \theta_1}} \qquad \text{for some} \quad a, b \les N^{\theta_1}.
$$

In order to evaluate the number of such integers we need the following version of Erd\"{o}s-Tur\'{a}n-Koksma inequality (see~\cite[Chapter~2]{KN74}, \cite{Koksma}, \cite{Szusz}):

\begin{lemma} \label{Erdos-Turan-Koksma}
	Let $\bm{x}_1, \ldots,  \bm{x}_N$ be a sequence of points in $\R^m$. Then for any integer $H > 0$, we have
	$$
		D_N (\bm{x}) \les C_m \biggl( \frac{1}{H} + \sum_{0 < |\bm{h}| \les H} \frac{1}{r(\bm{h})} \biggl| \frac{1}{N} \sum_{n \les N} e\bigl( \langle \bm{h}, \bm{x}_n \rangle \bigr) \biggr| \biggr),
	$$ where $|\bm{h}| = \max_{j \les m} |h_j|$,
	$$
		r(\bm{h}) = \prod_{j=1}^m \max(1, |h_j|),
	$$ and the constant $C_m$ only depends on the dimension $m$. 
\end{lemma}

Let us enumerate the numbers $\{ (n/m)^c \}$ as $y_1, \ldots, y_M$ with $M = 2^{2k}$. By assumption,	 $N^{\theta} \ll M \ll N^{2 \theta}$. Clearly there are  $\ll N^{2\theta_1}$ rational numbers $a/b$ satysfying $a < b \ll N^{\theta_1}$. Let us fix one of them. By Lemma~\ref{Erdos-Turan-Koksma},
$$
	\sum_{i=1}^M \mathbbm{1}\left( \left| y_i - \frac{a}{b} \right| \les N^{-1/2} \right) \ll \frac{M}{N^{1/2}} + \frac{M}{H} + \sum_{h \les H} \frac{1}{h} \biggl| \sum_{i=1}^M e\left( h y_i \right) \biggr|.
$$

The exponential sum in the last term can be estimated by the theorem of van der Corput (see, for example,~\cite[Theorem~8.20]{IK}):
$$
	\sum_{m = 2^k+1}^{2^{k+1}} \sum_{n = 2^k+1}^{2^{k+1}} e\left( h \left(\frac{n}{m}\right)^c \right) \ll 2^k 2^{k(1-\kappa)}	= M^{1 - \kappa/2}
$$ with some $\kappa = \kappa(c, H) > 0$. Thus, the total number of ``bad'' pairs of integers does not exceed
$$
	N^{2\theta_1} \left( \frac{M}{N^{1/2}} + \frac{M}{H} + M^{1- \kappa/2} \log H \right) \ll M^{1 - \kappa_0}
$$ for some $\kappa_0 > 0$ as soon as we choose $\theta_1, H$, so that $\kappa \theta / 2 > 3\theta_1$, $H \ges N^{3\theta_1}$. Note that $\theta_1$ can be arbitrarily close to zero. By the prime number theorem, there are $\gg M / (\log M)^2$ pairs of primes in $(2^k, 2^{k+1}]$, so most of them should satisfy~\eqref{restrict_irrat}.

\section{Proof of Proposition~\ref{prop2}}
\label{pf_prop2}

The proof goes similarly to Proposition~2.8 in~\cite{Mull-Spieg}. First, we need to approximate the Piatetski-Shapiro numbers by the Beatty sequence. This can be done via a linear approximation. It is given in the following lemma, the proof of which can be found in~\cite{Spieg14}:

\begin{lemma} \label{Spieg_lemma}
	Let $a, b$ be integers such that $a < b$, and set $K = b-a$. Assume that $f: [a, b] \to \R$ is twice differentiable and $|f''| \les B$. Then for all $\alpha \in f'([a,b])$ one has
	$$
		\# \bigl\{ n \in (a, b]: \floor{f(n)} \neq \floor{\alpha n + f(a) - \alpha a} \bigr\} \les 2B K^3 + KD_K (\alpha n). 
	$$
\end{lemma}

We also need a mean-square discrepancy estimate for the sequence $\{ \alpha n \}$ (see~\cite[Lemma~4.4]{Spieg}):

\begin{lemma} \label{mean-square-discr}
	For any integer $N \ges 3$ we have
	$$
	\int_0^1 D_N (\alpha n) d\alpha \ll \frac{(\log N)^2}{N},
	$$ where the implied constant in the estimate is absolute.
\end{lemma}

We start by splitting the interval $(N, 2N]$ to $L$ subintervals of length $K$ each, so that $LK = N$. Denote them as $(a_i, a_{i+1}]$, where $a_i = N + iK$. By the triangle inequality,
\begin{multline} \label{triangle_ineq}
	\left| \# \bigl\{ n \in (a_i, a_{i+1}]: t(\floor{(pn)^c}) = 1, \ t(\floor{(qn)^c}) = 1 \bigr\} - \frac{K}{4} \right| \les \\
	T_1 (\alpha, i) + T_2 (\alpha, i) + T_3 (\alpha, i),
\end{multline} where
\begin{multline*}
	T_1 (\alpha, i) := \biggl| \# \bigl\{ n \in (a_i, a_{i+1}]: t(\floor{(pn)^c}) = 1, \ t(\floor{(qn)^c}) = 1 \bigr\} - \\
	\# \bigl\{ n \in (a_i, a_{i+1}]: t(\floor{(pn)^c}) = 1, \ t(\floor{\gamma \alpha n + (qa_i)^c - \gamma \alpha a_i}) = 1 \bigr\} \biggr|,
\end{multline*}
\begin{multline*}
	T_2 (\alpha, i) := \biggl| \# \bigl\{ n \in (a_i, a_{i+1}]: t(\floor{(pn)^c}) = 1, \ t(\floor{\gamma \alpha n + (qa_i)^c - \gamma \alpha a_i}) = 1 \bigr\} - \\
	\# \bigl\{ n \in (a_i, a_{i+1}]: t(\floor{\alpha n + (pa_i)^c - \alpha a_i}) = 1, \ t(\floor{\gamma \alpha n + (qa_i)^c - \gamma \alpha a_i}) = 1 \bigr\} \biggr|,
\end{multline*}
\begin{multline*}
	T_3 (\alpha, i) := \biggl| \# \bigl\{ n \in (a_i, a_{i+1}]: t(\floor{\alpha n + (pa_i)^c - \alpha a_i}) = 1, \\
	t(\floor{\gamma \alpha n + (qa_i)^c - \gamma \alpha a_i}) = 1 \bigr\} - \frac{K}{4} \biggr|.
\end{multline*} Let us denote $f(n) := (2^k n)^c$. Next, we integrate both sides of the inequality~\eqref{triangle_ineq} over $\alpha \in [f'(a_i), f'(a_{i+1})]$, divide by $(f'(a_{i+1}) - f'(a_i))$, and sum over $0 \les i < L$. We obtain
\begin{multline} \label{three_T}
	\left| \# \bigl\{ n \in (a_0, a_L]: t(\floor{(pn)^c}) = 1, \ t(\floor{(qn)^c}) = 1 \bigr\} - \frac{LK}{4} \right| \les \\ \sum_{0 \les i < L} \frac{1}{f'(a_{i+1}) - f'(a_i)} \int_{f'(a_i)}^{f'(a_{i+1})} \bigl( T_1 (\alpha, i) + T_2 (\alpha, i) + T_3 (\alpha, i) \bigr) d\alpha.
\end{multline} 

By Lemma~\ref{Spieg_lemma},
\begin{equation} \label{contribution}
	T_j (\alpha, i) \ll 2^{kc} N^{c-2} K^3 + K D_K (\alpha n) \qquad \text{for } j = 1, 2.
\end{equation} We have omitted in $T_1 (\alpha, i)$ and $T_2 (\alpha, i)$ the conditions 
$$
	t(\floor{(pn)^c}) = 1 \qquad  \text{and} \qquad t(\floor{\gamma \alpha n + (qa_i)^c - \gamma \alpha a_i}) = 1
$$ correspondingly, since they only decrease the bound. The contribution to the right side of~\eqref{three_T} from the first term on the right side of~\eqref{contribution} is bounded by
$$
	2^{kc} N^{c-2} K^3 L = 2^{kc} N^{c-1} K^2.
$$

Next,
$$
	|f'(2N) - f'(N)| = \biggl| \sum_{0 \les i < L} f'(a_{i+1}) - f'(a_i) \biggr| \asymp L |f'(a_{i+1}) - f'(a_i)|
$$ for any fixed $i$. So we get
$$
	\frac{1}{f'(a_{i+1}) - f'(a_i)} \ll \frac{N}{K} \frac{1}{f'(2N) - f'(N)} \asymp \frac{N}{K 2^{kc} N^{c-1}}.
$$ Then, by Lemma~\ref{mean-square-discr}, the contribution to~\eqref{three_T} from the second term in~\eqref{contribution} is bounded by
\begin{multline*}
	K \sum_{0 \les i < L} \frac{1}{f'(a_{i+1}) - f'(a_i)} \int_{f'(a_i)}^{f'(a_{i+1})} D_K (\alpha n) d\alpha \ll \\
	\frac{N}{2^{kc} N^{c-1}} \sum_{0 \les i < L} \int_{f'(a_i)}^{f'(a_{i+1})} D_K (\alpha n) d\alpha \ll \\
	\frac{N}{2^{kc} N^{c-1}} \bigl( f'(2N) - f'(N) + 1 \bigr) \int_0^1 D_K (\alpha n) d\alpha \ll \\ \frac{N \cdot 2^{kc} N^{c-1}}{2^{kc} N^{c-1}} \frac{(\log K)^2}{K} = N \frac{(\log K)^2}{K}.
\end{multline*}

The contribution from $T_3 (\alpha, i)$ is trivially bounded by
$$
	\sum_{0 \les i < L} \frac{1}{f'(a_{i+1}) - f'(a_i)} \int_{f'(a_i)}^{f'(a_{i+1})} T_3 (\alpha, i) d\alpha \ll N J(N, K).
$$ Finally,
	\begin{multline*}
	\biggl| \frac{1}{N} \# \bigl\{ n \in (N, 2N]: t(\floor{(pn)^c}) = 1, \ t(\floor{(qn)^c}) = 1 \bigr\} - \frac{1}{4} \biggr| \ll \\
	2^{kc} N^{c-2} K^2 + \frac{(\log N)^2}{K} + J(N, K)
\end{multline*} as desired.

\section{Proof of Proposition~\ref{prop3}: cutting of digits}
\label{cutting}
 
In this section we approximate the functions $s(\floor{\alpha n + \beta_1})$ and $s(\floor{\gamma \alpha n + \beta_2})$ by the truncated functions $s_{\rho}(\floor{\alpha n + \beta_1})$ and $s_{\rho}(\gamma \alpha n + \beta_2)$ with some integer $\rho > 0$, such that $2^{\rho}$ is small compared to $N$. We need the following two versions of van der Corput inequality:

\begin{lemma} \label{generalized_Corput}
	Let $I$ be a finite interval in $\Z$ containing $N$ integers and let $z_n \in \C$ for $n \in I$. Then
	
	i) For all integers $R \ges 1$ we have
	$$
		\biggl| \sum_{n \in I} z_n \biggr|^2 \les \frac{N + R - 1}{R} \sum_{0 \les |r| < R} \left( 1 - \frac{|r|}{R} \right) \sum_{\substack{n \in I \\ n + r \in I}} z_{n+r} \overline{z_n}.
	$$
	
	ii) Assume that $K \subset \N$ is a finite nonempty set. Then
	$$
		\biggl| \sum_{n \in I} z_n \biggr|^2 \les \frac{N + \max K - \min K}{|K|^2} \sum_{k_1, k_2 \in K} \sum_{\substack{n \in \Z \\ n, n+k_1-k_2 \in I}} z_n \overline{z_{n+k_1-k_2}}.
	$$
\end{lemma} The proof of the first part can be found in~\cite[Lemme~2.7]{MR2009}, for the second part see~\cite[Proposition~6.14]{DMS_126}.

We also need the following statement concerning the distribution of the numbers $\{ \alpha n + \beta \}$ (see~\cite[Lemma~4.3]{Spieg}):

\begin{lemma} \label{Spieg_frac_parts}
	Let $J$ be an interval in $\R$ containing $N$ integers and let $\alpha$ and $\beta$ be real numbers. Further, assume that $t$ and $T$ are integers such that $0 \les t < T$. Then
	$$
		\# \left\{ n \in J: \frac{t}{T} \les \{ \alpha n + \beta \} < \frac{t+1}{T}  \right\} = \frac{N}{T} + O\bigl( N D_N (\alpha n) \bigr),
	$$ where the implied constant is absolute.
\end{lemma}

Let us denote the expression in the left side of~\eqref{expsum} by $S_0 = S_0 (N, D)$. Then trivially $|S_0| \les DN$. By Cauchy inequality,
$$
	|S_0|^2 \les D \int_D^{2D} \max_{\beta_1, \beta_2 \ges 0} \biggl| \sum_{n \les N} e\left( \frac{1}{2} s(\floor{\alpha n + \beta_1}) + \frac{1}{2} s \left( \floor{\gamma \alpha n + \beta_2}\right) \right) \biggr|^2 d\alpha.
$$ Then, by the first part of Lemma~\ref{generalized_Corput},
\begin{multline*}
	|S_0|^2 \les \frac{DN}{R_0} \int_D^{2D} \max_{\beta_1, \beta_2 \ges 0} \sum_{r_0 \les R_0} \biggl| \sum_{\substack{n \les N \\ n+r_0 \les N}} e \biggl( \frac{1}{2} s(\floor{\alpha n + \alpha r_0 + \beta_1}) - \frac{1}{2} s(\floor{\alpha n + \beta_1}) + \\
	\frac{1}{2} s(\floor{\gamma \alpha n + \gamma \alpha r_0 + \beta_2}) - \frac{1}{2} s(\floor{\gamma \alpha n + \beta_2}) \biggr) \biggr| d\alpha + O\left( \frac{(DN)^2}{R_0}  \right),
\end{multline*} where the error term comes from $r=0$.

To truncate the sum of digits function we apply the version of carry property given by the following lemma (see~\cite[Lemma~4.5]{Spieg}):

\begin{lemma}[carry propagation lemma] \label{propagation}
	Let $N, r, \lambda$ be nonnegative integers and $\alpha > 0$, $\beta \ges 0$ real numbers. Then
	\begin{multline*}
		\# \bigl\{ n \les N : s(\floor{\alpha n + \alpha r + \beta}) - s(\floor{\alpha n + \beta}) \neq \\
		s_{\lambda}(\floor{\alpha n + \alpha r + \beta}) - s_{\lambda}(\floor{\alpha n + \beta})  \bigr\} \les r \left( \frac{\alpha N}{2^{\lambda}} + 2 \right).
	\end{multline*}
\end{lemma} Applying this lemma, we get
\begin{multline*}
	\# \bigl\{ n \les N : s(\floor{\alpha n + \alpha r_0 + \beta_1}) - s(\floor{\alpha n + \beta_1}) \neq \\
	s_{\lambda}(\floor{\alpha n + \alpha r_0 + \beta_1}) - s_{\lambda}(\floor{\alpha n + \beta_1})  \bigr\} \les r_0 \left( \frac{\alpha N}{2^{\lambda}} + 2 \right).
\end{multline*} The same clearly holds true for $\alpha$, $\beta_1$ replaced by $\gamma \alpha$, $\beta_2$. Thus, by the union bound,
\begin{multline*}
	|S_0|^2 \les \frac{DN}{R_0} \int_D^{2D} \max_{\beta_1, \beta_2 \ges 0} \sum_{r_0 \les R_0} \biggl| \sum_{\substack{n \les N \\ n+r_0 \les N}} e \biggl( \frac{1}{2} s_{\lambda} (\floor{\alpha n + \alpha r_0 + \beta_1}) - \frac{1}{2} s_{\lambda}(\floor{\alpha n + \beta_1}) + \\
	\frac{1}{2} s_{\lambda}(\floor{\gamma \alpha n + \gamma \alpha r_0 + \beta_2}) - \frac{1}{2} s_{\lambda}(\floor{\gamma \alpha n + \beta_2}) \biggr) \biggr| d\alpha + O\left( (DN)^2 \bigl( \frac{1}{R_0} +  \frac{R_0 D}{2^{\lambda}} + \frac{R_0}{N} \bigr) \right).
\end{multline*} For brevity, we denote the error term in the last expression by $E_1^2$:
\begin{equation} \label{for_E_1}
	E_1^2 := (DN)^2 \left( \frac{1}{R_0} + \frac{R_0 D}{2^{\lambda}} + \frac{R_0}{N} \right).
\end{equation}

Applying Cauchy inequality followed by $m$ applications of the second part of Lemma~\ref{generalized_Corput}, we obtain
\begin{multline*}
	|S_0|^{2^{m+1}} \ll \frac{(DN)^{2^{m+1}-1}}{R_0} \int_D^{2D} \max_{\beta_1, \beta_2 \ges 0} \sum_{r_0 \les R_0} \frac{1}{|K_1 (\alpha)|^2 \ldots |K_m (\alpha)|^2} \sum_{k_1, k_1' \in K_1 (\alpha)} \ldots \\ 
	\sum_{k_m, k_m' \in K_m (\alpha)}
	\biggl| \sum_{n \les N} e\biggl( \frac{1}{2} \sum_{\substack{\varepsilon_0, \ldots, \varepsilon_m \\ \in \{0, 1\}}}  \biggl( s_{\lambda}\big(\floor{\alpha n + \beta_1 + \varepsilon_0 \alpha r_0 + \alpha \langle \bm{\varepsilon}, \bm{k} - \bm{k}' \rangle }\big) + \\ s_{\lambda}\big(\floor{\gamma \alpha n + \beta_2 + \varepsilon_0 \gamma \alpha r_0 + \gamma \alpha \langle \bm{\varepsilon}, \bm{k} - \bm{k}' \rangle}\big) \biggr) \biggr) \biggr| d\alpha + E_1^{2^{m+1}} + E_2^{2^{m+1}},
\end{multline*} where $\bm{\varepsilon} = (\varepsilon_1, \ldots, \varepsilon_m)$, $\bm{k} - \bm{k}' = (k_1 - k_1', \ldots, k_m - k_m')$, $ \max_i|K_i (\alpha)| \ll B$ for all $D \les \alpha < 2D$ and some $0 < B \les N$, which will be chosen later. Here
\begin{equation} \label{for_E_2}
	E_2^{2^{m+1}} := (DN)^{2^{m+1}} \frac{B}{N}.
\end{equation} Note that the additional restrictions $n+r_0 \les N, n+k_i-k_i' \les N$ were removed, since the corresponding error terms are absorbed by $E_2$.

Each application of van der Corput inequality corresponds to cutting away of $\mu$ leading digits of $\floor{\alpha n + \beta_1 + \ldots}$ and $\floor{\gamma \alpha n + \beta_2 + \ldots}$, where we require $\mu > 0$ and $\lambda - m \mu > \sigma > 0$. The numbers $\mu \gg \sigma \ges 10$ will be chosen later. We choose the subsets $K_i (\alpha)$ in the following way:
\begin{multline*}
	K_i (\alpha) := \bigl\{ k \les B : \delta_j (\floor{\alpha k}) = 0, \ \delta_j (\floor{\gamma \alpha k}) = 0 \\
	\text{for all } j \in (\lambda - i \mu - \sigma, \ \lambda - (i-1) \mu]  \bigr\},
\end{multline*} where $\delta_j(n)$ is a binary digit of $n$ in the $j$-th position.

Next, we show that for most $n \les N$ the digits of $\floor{\alpha n + \beta_1 + \ldots }$ (when $\varepsilon_i = 1$) and $\floor{\alpha n + \beta_1 + \ldots}$ (when $\varepsilon_i = 0$) in the positions in $(\lambda - i\mu, \ \lambda - (i-1)\mu]$ are the same, and, at the same time, this property holds true for the second factor $\floor{\gamma \alpha n + \beta_2 + \ldots}$ as well. Indeed, for a real number $t$ one can write
$$
	\floor{t + \alpha \varepsilon_i (k_i - k_i')} = \floor{t} + \floor{\alpha \varepsilon_i k_i} - \floor{\alpha \varepsilon_i k_i'} + c_i,
$$ where $c_i$ is an integer depeding on $t, \alpha, \varepsilon_i, k_i, k_i'$, such that $|c_i| \les 2$. In order to avoid propagation from the right, we require $\floor{t}$ to have at least one $1$ and at least one $0$ in the positions in $(\lambda - i\mu - \sigma + 2, \ \lambda - i \mu]$. Then adding the terms $\floor{\alpha k_i}$, $\floor{\alpha k_i'}$, and $c_i$, does not change the digits of $\floor{t}$ in $(\lambda - i \mu, \ \lambda - (i-1) \mu]$. Thus, the digits of $\floor{t + \alpha (k_i - k_i')}$ and $\floor{t}$ are the same in the interval $(\lambda - i \mu, \ \lambda - (i-1) \mu]$, and so 
$$
	s_{\lambda - (i-1)\mu}(\floor{\alpha n + \beta_1 + \ldots}) = s_{\lambda - i\mu}(\floor{\alpha n + \beta_1 + \ldots}).
$$ By our choice of $K_i (\alpha)$, the same property holds for the second factor with 
$$
	\floor{\gamma \alpha n + \beta_2 + \ldots}.
$$

It only remains to show that for a ``typical'' value of $\alpha \in [D, 2D)$ the numbers $\floor{\alpha n + \beta_3}$ and $\floor{\gamma \alpha n + \beta_4}$ with 
$$
	\beta_3 = \beta_1 + \alpha \varepsilon_0 r_0 + \alpha \langle \bm{\varepsilon}, \bm{k} - \bm{k}' \rangle - \alpha \varepsilon_i (k_i - k_i')
$$ and
$$
	\beta_4 = \beta_2 + \gamma \alpha \varepsilon_0 r_0 + \gamma \alpha \langle \bm{\varepsilon}, \bm{k} - \bm{k}' \rangle - \gamma \alpha \varepsilon_i (k_i - k_i')	
$$ satisfy the aforementioned property for most $n \les N$. Let us estimate the contribution from the exceptional set of $n \les N$. For each such $n$ there exists a pattern $(\varepsilon_0^{(n)}, \ldots, \varepsilon_{i-1}^{(n)}, 0, \varepsilon_{i+1}^{(n)}, \ldots, \varepsilon_m^{(n)})$ such that either $\floor{\alpha n + \beta_3}$ or $\floor{\gamma \alpha n + \beta_4}$ have only $0$'s or only $1$'s in the positions in $(\lambda - i\mu - \sigma + 2, \lambda - i\mu]$. By the union bound it is enough to consider the case when $\floor{\alpha n + \beta_3}$ has only $0$'s in this interval. This condition can be rewritten in terms of the fractional parts as
$$
	\left\{ \frac{\alpha n + \beta_3}{2^{\lambda - i\mu}}  \right\} \in \left[ 0, \frac{1}{2^{\sigma - 2}} \right).
$$ By Lemma~\ref{Spieg_frac_parts}, the number of such $n \les N$ is bounded from above by
$$
	\frac{N}{2^{\sigma - 2}} + O\left( N D_N \left( \frac{\alpha n}{2^{\lambda - i\mu}} \right) \right).
$$ Then, by Lemma~\ref{mean-square-discr},
$$
	\int_D^{2D} D_N \left( \frac{\alpha n}{2^{\lambda - i\mu}} \right) d\alpha \ll \frac{D}{N} (\log N)^2.
$$ Using the fact that there are at most $2^{m+1}$ terms in the union bound, we find that the total contribution to $|S_0|^{2^{m+1}}$ from the exceptional set of $n \les N$ does not exceed
\begin{equation} \label{for_E_3}
	2^{m+1} (DN)^{2^{m+1}} \left( \frac{1}{2^{\sigma-2}} + \frac{(\log N)^2}{N} \right) =: E_3^{2^{m+1}}.
\end{equation}

Finally, using the notation $\rho := \lambda - m \mu$, we get
\begin{multline} \label{bound_after_sec4}
	|S_0|^{2^{m+1}} \les \frac{(DN)^{2^{m+1}-1}}{R_0} \int_D^{2D} \max_{\beta_1, \beta_2 \ges 0} \sum_{r_0 \les R_0} \frac{1}{|K_1 (\alpha)|^2 \ldots |K_m (\alpha)|^2} \sum_{k_1, k_1' \in K_1 (\alpha)} \ldots \\
	\sum_{k_m, k_m' \in K_m (\alpha)}
	\biggl| \sum_{n \les N} e\biggl( \frac{1}{2} \sum_{\substack{\varepsilon_0, \ldots, \varepsilon_m \\ \in \{ 0, 1 \}}} \biggl( s_{\rho}\big(\floor{\alpha n + \beta_1 + \alpha \varepsilon_0 r_0 + \alpha \langle \bm{\varepsilon}, \bm{k} - \bm{k}' \rangle}\big) + \\ s_{\rho}\big(\floor{\gamma \alpha n + \beta_2 + \gamma \alpha \varepsilon_0 r_0 + \gamma \alpha \langle \bm{\varepsilon}, \bm{k} - \bm{k}' \rangle}\big) \biggr) \biggr) \biggr| d\alpha +
	E_1^{2^{m+1}} + E_2^{2^{m+1}} + E_3^{2^{m+1}}.
\end{multline}

\section{Proof of Proposition~\ref{prop3}: discrepancy for the sum over $n$}
\label{Koksma_on_n}

Due to the restriction on the rational approximations of $\gamma = (q/p)^c$ (see~\eqref{restrict_irrat}), the sequence $(\alpha n \bmod{1}, \gamma \alpha n \bmod{1})$ has a low discrepancy on the unit square. Using Koksma-Hlawka inequality we will approximate the sum over $n$ in~\eqref{bound_after_sec4} by an integral over~$[0, 1]^2$: 
\begin{multline*}
	\frac{1}{N} \sum_{n \les N} e\biggl( \frac{1}{2} \sum_{\substack{\varepsilon_0, \ldots, \varepsilon_m \\ \in \{ 0, 1 \}}}  \biggl( s_{\rho}\big(\floor{\alpha n + \beta_1 + \alpha \varepsilon_0 r_0 + \alpha \langle \bm{\varepsilon}, \bm{k} - \bm{k}' \rangle}\big) + \\ 
	s_{\rho}\big(\floor{\gamma \alpha n + \beta_2 + \gamma \alpha \varepsilon_0  r_0 + \gamma \alpha \langle \bm{\varepsilon}, \bm{k} - \bm{k}' \rangle}\big) \biggr) \biggr) \approx \\ 
	\int_0^1 \int_0^1 e\biggl( \frac{1}{2}  \sum_{\substack{\varepsilon_0, \ldots, \varepsilon_m \\ \in \{0, 1\}}} \biggl( g_{\rho} \bigl( x + \beta_1 + \frac{\alpha \varepsilon_0 r_0}{2^{\rho}} + \frac{\alpha}{2^{\rho}} \langle \bm{\varepsilon}, \bm{k} - \bm{k}' \rangle \bigr) + \\
	g_{\rho} \bigl(y + \beta_2 + \frac{\gamma \alpha \varepsilon_0  r_0}{2^{\rho}} + \frac{\gamma \alpha}{2^{\rho}} \langle \bm{\varepsilon}, \bm{k} - \bm{k}' \rangle \bigr) \biggr) \biggr) dxdy
\end{multline*} for a certain 1-periodic function $g_{\rho} (x)$ with a bounded total variation. Indeed, note that $s_{\rho} (\floor{t})$ can be expressed in terms of values of a piecewise continuous 1-periodic function $g_{\rho}(x)$. The values of $g_{\rho} (x)$ are also in $\{0, 1\}$ and depend only on the fractional part of $t / 2^{\rho}$:
$$
	s_{\rho} (\floor{t}) =: g_{\rho} \left( \frac{\floor{t}}{2^{\rho}} \right) = g_{\rho} \left( \frac{t}{2^{\rho}} \right).
$$ Then
$$
	s_{\rho} \bigl(\floor{\alpha n + \beta_1 + \alpha \varepsilon_0 r_0 + \alpha \langle \bm{\varepsilon}, \bm{k} - \bm{k}' \rangle}\bigr) = g_{\rho} \left( \frac{\alpha n + \beta_1 + \alpha \varepsilon_0 r_0 + \alpha \langle \bm{\varepsilon}, \bm{k} - \bm{k}' \rangle}{2^{\rho}} \right).
$$ 

To apply the Koksma-Hlawka inequality we need to introduce the concept of the \textit{Hardy-Krause total variation}. Let $f(\bm{x})$ be a function defined on the cube $[0, 1]^m$. We define a \textit{partition} $\bm{P}$ of $[0, 1]^m$ as the set of $m$ sequences $\eta_0^{(j)}, \eta_1^{(j)}, \ldots, \eta_{n_j}^{(j)}$ such that $0 = \eta_0^{(j)} \les \eta_1^{(j)} \les \ldots \les \eta_{n_j}^{(j)} = 1$ for all $j = 1, \ldots, m$, and the difference operator $\Delta_j$,
\begin{multline*}
	\Delta_j f (x^{(1)}, \ldots,  x^{(j-1)}, \eta_i^{(j)}, x^{(j+1)}, \ldots, x^{(m)}) = \\
	f (x^{(1)}, \ldots,  x^{(j-1)}, \eta_{i+1}^{(j)}, x^{(j+1)}, \ldots, x^{(m)}) -
	f (x^{(1)}, \ldots,  x^{(j-1)}, \eta_i^{(j)}, x^{(j+1)}, \ldots, x^{(m)})
\end{multline*} for $0 \les i < n_j$. Then the \textit{Vitali total variation} of $f(\bm{x})$ can be defined as
$$
	V^{(m)} (f) = \sup_{\bm{P}} \sum_{i_1 = 0}^{n_1 - 1} \ldots \sum_{i_m = 0}^{n_m - 1} \left| \Delta_1 \ldots \Delta_k f(\eta_{i_1}^{(1)}, \ldots, \eta_{i_m}^{(m)}) \right|,
$$ where the supremum is taken over all partitions of $[0, 1]^m$. If a given function does not depend on one of the variables, the Vitali total variation is not useful as it may not reflect the actual behavior of the function. A more complete picture is given by the \textit{Hardy-Krause total variation}, which reflects the behavior of $f(\bm{x})$ on all faces of $[0, 1]^m$. Let $\bm{I} = \{ l_1, \ldots, l_{|\bm{I}|} \} \subset \{1, \ldots, s \}$ be a subset of indices. Then the \textit{Hardy-Krause total variation} can be defined as
$$
	V_{HK}^{(m)} (f) = \sup_{\bm{I}} \sup_{\bm{P}} \sum_{i_1 = 0}^{n_1 - 1} \ldots \sum_{i_m = 0}^{n_m - 1} \left| \Delta_{l_1} \ldots \Delta_{l_{\bm{I}}} f(\eta_{i_1}^{(1)}, \ldots, \eta_{i_m}^{(m)}) \right|.
$$

Now we are ready to state the Koksma-Hlawka inequality. The proof of this version can be found in~\cite[Theorem~5.5]{KN74}):

\begin{lemma} \label{Koksma-Hlawka}
	Let $[0, 1]^m$ be the $m$-dimensional unit cube, and let $f(\bm{x})$ be a function with bounded Hardy-Krause total variation $V_{HK}^{(m)}(f)$ on $[0, 1]^m$. Then for any sequence $\bm{x}_1, \ldots, \bm{x}_N$ in $[0, 1]^m$ one has
	$$
		\biggl| \frac{1}{N} \sum_{i=1}^N f(\bm{x}_i) - \int_{[0, 1]^m} f(\bm{u}) d\bm{u} \biggr| \les V_{HK}^{(m)} (f) D_N (\bm{x}).
	$$
\end{lemma}

We apply Lemma~\ref{Koksma-Hlawka} with $s=2$ and
\begin{multline*}
	f(\bm{x}_n) = f\left( \frac{\alpha n}{2^{\rho}}, \frac{\gamma \alpha n}{2^{\rho}}  \right) :=
	e\biggl( \frac{1}{2} \sum_{\substack{\varepsilon_0, \ldots, \varepsilon_m \\ \in \{ 0, 1 \}}} \biggl( g_{\rho} \left(\frac{\alpha n + \beta_1 + \alpha \varepsilon_0 r_0 + \alpha \langle \bm{\varepsilon}, \bm{k} - \bm{k}' \rangle}{2^{\rho}}\right) + \\
	g_{\rho} \left(\frac{\gamma \alpha n + \beta_2 + \gamma \alpha \varepsilon_0 r_0 + \gamma \alpha \langle \bm{\varepsilon}, \bm{k} - \bm{k}' \rangle}{2^{\rho}}\right) \biggr) \biggr).
\end{multline*} The total variation of $f$ can be crudely estimated as
$$
	V_{HK}^{(2)}(f) \ll 2^{m+\rho},
$$ since the function $g_{\rho}$ changes the sign $2^{\rho}$ times in $[0, 1]$, and there are $2^{m+1}$ different vectors $\bm{\varepsilon}$. So it only remains to evaluate the discrepancy $D_N (\alpha n / 2^{\rho}, \gamma \alpha n / 2^{\rho})$.

In the next three sections we will deal with various discrepancies. To estimate them, we need upper bounds for the averages of linear exponential sums. They are given in the following lemma:

\begin{lemma} \label{expsums_estimates}
	Let $D, N, t$ be real numbers such that $D, N > 10$, and $D|t| \ges N^{-1}$. Then one has
	$$
		W := \int_D^{2D} \biggl| \sum_{n \les N} e\bigl( \alpha t n \bigr) \biggr| d\alpha \ll \max \left( D \log N, \frac{1}{|t|} \right).	
	$$  
\end{lemma}

\begin{proof}
	By the linear Weyl's bound, 
	$$
		W \les \int_D^{2D} \min \left( N, \frac{1}{\norm{\alpha t}} \right)	d\alpha.
	$$ Substituting $u = \alpha t$, we get
	$$
		W \les \frac{1}{t} \int_{Dt}^{2Dt} \min \left(N, \frac{1}{\norm{u}} \right)du.
	$$ If $2D|t| > 1/2$, then
	$$
		W \ll \frac{1}{t} Dt \int_0^{1/2} \min \left(N, \frac{1}{u} \right) du \ll D \log N.
	$$ Otherwise,
	$$
		W \ll \frac{1}{t} \int_{Dt}^{2Dt} \frac{1}{u} du \ll \frac{1}{|t|}
	$$ as desired.
\end{proof}

By Lemma~\ref{Erdos-Turan-Koksma},
\begin{equation} \label{discrepancy}
	D_N \left( \frac{\alpha n}{2^{\rho}}, \frac{\gamma \alpha n}{2^{\rho}}  \right) \les C \biggl( \frac{1}{H} + \sum_{\substack{|h_1|, |h_2| \les H \\ (h_1, h_2) \neq (0, 0)}} \frac{1}{r(\bm{h})} \biggl| \frac{1}{N} \sum_{n \les N} e\left( h_1 \frac{\alpha n}{2^{\rho}} + h_2 \frac{\gamma \alpha n}{2^{\rho}} \right)  \biggr| \biggr)
\end{equation} for any positive integer $H$. Thus, the first term in the discrepancy estimate~\eqref{discrepancy} contributes to the bound~\eqref{bound_after_sec4} for $|S_0|^{2^{m+1}}$ at most
\begin{equation} \label{for_E_4}
	E_4^{2^{m+1}} := (DN)^{2^{m+1}-1} 2^{m+\rho} \frac{DN}{H}.
\end{equation} To estimate the contribution from the second term, we apply Lemma~\ref{expsums_estimates} ty the sum over $n \les N$ with $t = (h_1 + \gamma h_2) / 2^{\rho}$. Then the contribution to $|S_0|^{2^{m+1}}$ is bounded by
$$
	(DN)^{2^{m+1}-1} 2^{m+\rho} \sum_{\substack{|h_1|, |h_2| \les H \\ (h_1, h_2) \neq (0,0)}} \frac{1}{r(\bm{h})} \max \left( D \log N, \frac{2^{\rho}}{|h_1 + \gamma h_2|} \right).
$$ By the assumption~\eqref{restrict_irrat},
$$
	\left| \gamma - \frac{h_1}{h_2} \right| \ges N^{-1/2}
$$ for any $|h_2| \les N^{\theta_1}$, which clearly implies $|h_1 + \gamma h_2| \ges N^{-1/2}$. Since $D$ is larger than $N$, and $2^{\rho}$ is much smaller than $N$, the contribution from the second term on the right side of~\eqref{discrepancy} is estimated as
\begin{equation} \label{for_E_5}
	E_5^{2^{m+1}} := (DN)^{2^{m+1}} \frac{2^{m+\rho} (\log N)^3}{N}
\end{equation} as soon as $H \ll N^{\theta_1}$. Thus, 
\begin{multline*}
	|S_0|^{2^{m+1}} \les \frac{N(DN)^{2^{m+1}-1}}{R_0} \int_D^{2D} \max_{\beta_1, \beta_2 \ges 0} \sum_{r_0 \les R_0} \frac{1}{|K_1 (\alpha)|^2 \ldots |K_m (\alpha)|^2} \sum_{k_1, k_1' \in K_1 (\alpha)} \ldots \\ 
	\sum_{k_m, k_m' \in K_m (\alpha)}
	\biggl| \int_0^1 \int_0^1 e\biggl( \frac{1}{2}  \sum_{\substack{\varepsilon_0, \ldots, \varepsilon_m \\ \in \{0, 1\}}} \biggl( g_{\rho} \bigl( x + \beta_1 + \frac{\alpha \varepsilon_0 r_0}{2^{\rho}} + \frac{\alpha}{2^{\rho}} \langle \bm{\varepsilon}, \bm{k} - \bm{k}' \rangle \bigr) + \\
	g_{\rho} \bigl(y + \beta_2 + \frac{\gamma \alpha \varepsilon_0  r_0}{2^{\rho}} + \frac{\gamma \alpha}{2^{\rho}} \langle \bm{\varepsilon}, \bm{k} - \bm{k}' \rangle \bigr) \biggr) \biggr) dxdy \biggr| d\alpha + \sum_{j=1}^5 E_j^{2^{m+1}}.
\end{multline*}

Now since the variables $x$ and $y$ are independent, one can get rid of one of the factors $g_{\rho}(x+\ldots)$, $g_{\rho}(y+\ldots)$ by the Cauchy inequality:
\begin{multline*}
	|S_0|^{2^{m+2}} \les \frac{N (DN)^{2^{m+2}-1}}{R_0} \int_D^{2D} \max_{\beta \ges 0} \sum_{r_0 \les R_0} \frac{1}{|K_1 (\alpha)|^2 \ldots |K_m (\alpha)|^2} \sum_{k_1, k_1' \in K_1 (\alpha)} \ldots \\
	\sum_{k_m, k_m' \in K_m (\alpha)}
	\biggl| \int_0^1 e\biggl( \frac{1}{2}  \sum_{\substack{\varepsilon_0, \ldots, \varepsilon_m \\ \in \{0, 1\}}} g_{\rho} \bigl( x + \beta + \frac{\alpha \varepsilon_0 r_0}{2^{\rho}} + \frac{\alpha}{2^{\rho}} \langle \bm{\varepsilon}, \bm{k} - \bm{k}' \rangle \bigr) \biggr) dx \biggr|^2 d\alpha + \sum_{j=1}^5 E_j^{2^{m+2}}.
\end{multline*} Opening the square and changing the order of summation and integration, we obtain
\begin{multline*}
	|S_0|^{2^{m+2}} \ll \frac{N (DN)^{2^{m+2}-1}}{R_0} \int_D^{2D} \max_{\beta \ges 0} \int_0^1 \int_0^1 \sum_{r_0 \les R_0} \frac{1}{|K_1 (\alpha)|^2 \ldots |K_m (\alpha)|^2} \\
	\sum_{k_1, k_1' \in K_1 (\alpha)} \ldots \sum_{k_m, k_m' \in K_m (\alpha)} e\biggl( \frac{1}{2}  \sum_{\substack{\varepsilon_0, \ldots, \varepsilon_m \\ \in \{0, 1\}}} g_{\rho} \bigl( x + \beta + \frac{\alpha \varepsilon_0 r_0}{2^{\rho}} + \frac{\alpha}{2^{\rho}} \langle \bm{\varepsilon}, \bm{k} - \bm{k}' \rangle \bigr) + \\
	g_{\rho} \bigl( y + \frac{\alpha \varepsilon_0 r_0}{2^{\rho}} + \frac{\alpha}{2^{\rho}} \langle \bm{\varepsilon}, \bm{k} - \bm{k}' \rangle \bigr) \biggr) dx dy d\alpha + \sum_{j=1}^5 E_j^{2^{m+2}},
\end{multline*} which gives
\begin{multline*}
	|S_0|^{2^{m+2}} \ll \frac{N (DN)^{2^{m+2}-1}}{R_0} \int_D^{2D} \max_{\beta \ges 0} \int_0^1 \int_0^1 \sum_{r_0 \les R_0} \frac{1}{|K_1 (\alpha)|^2 \ldots |K_m (\alpha)|^2} \\ 
	\sum_{k_1, k_1' \in K_1 (\alpha)} \ldots \sum_{k_m, k_m' \in K_m (\alpha)} e\biggl( \frac{1}{2}  \sum_{\substack{\varepsilon', \varepsilon_0, \ldots, \varepsilon_m \\ \in \{ 0, 1 \}}} g_{\rho} \bigl(x + \beta + \varepsilon'y + \frac{\varepsilon_0 \alpha r_0}{2^{\rho}} + \frac{\alpha}{2^{\rho}} \langle \bm{\varepsilon}, \bm{k} - \bm{k}' \rangle \bigr) \biggr) dx dy d\alpha + \\
	\sum_{j=1}^5 E_j^{2^{m+2}}.
\end{multline*}

\section{Proof of Proposition~\ref{prop3}: discrepancy for the sums over $r_0$ and most of $k_i \in K_i (\alpha)$}
\label{Koksma_on_k_first}

We apply a similar strategy to the sums over $r_0 \les R_0, \ k_1 \in K_1 (\alpha), \ldots, k_{m-l} \in K_{m-l} (\alpha)$, where the integer $l$ is much smaller than $m$. The precise value will be chosen at the end of the proof. Since we have removed the factor containing $\gamma$, we only need a 1-dimensional version of the Koksma-Hlawka inequality.

We start with the sum over $r_0 \les R_0$. By Lemma~\ref{Koksma-Hlawka},
$$
	\biggl| \frac{1}{R_0} \sum_{r_0 \les R_0} f_{\rho} \left( \frac{\alpha r_0}{2^{\rho}} \right) - \int_0^1 f_{\rho}(x_0) dx_0 \biggr| \ll 2^{m + \rho} D_{R_0} \left( \frac{\alpha r_0}{2^{\rho}} \right),
$$ By Lemma~\ref{mean-square-discr} the contribution from this error term to $|S_0|^{2^{m+2}}$ is bounded by
\begin{equation} \label{for_F_0}
	F_0^{2^{m+2}} := (DN)^{2^{m+2}} \frac{2^{m + \rho}}{R_0} (\log N)^2
\end{equation} under the restriction $R_0 \ll N$.

Treating the sums over $k_1, \ldots, k_{m-l}$ is a little more tricky as one needs to evaluate the discrepancy among a more complicated subsequence. We first state a more precise version of Erd\"{o}s-Tur\'{a}n theorem for 1-dimensional case:

\begin{lemma} \label{Vinogradov_cup}
	Let $a$ and $b$  be real numbers such that $0 \les a < b \les 1$. Then the indicator function of the interval $[a, b]$ admits the expansion
	$$
		\mathbbm{1}_{[a, b]} (x) = (b-a) + \frac{\Theta}{H} + \sum_{0 < |h| \les H} \frac{\Theta_h}{|h|} e\bigl( hx \bigr)
	$$ and all large enough $H > 0$. Here $\Theta, \Theta_h$ are some complex numbers such that $|\Theta|, |\Theta_h|$ are bounded by an absolute constant.  
\end{lemma} The proof of this lemma is standard, and can be done, for example, via Vaaler polynomials~\cite[Theorem~A.6]{Graham_Kolesnik} or ``Vinogradov cups''~\cite[Chapter~1]{Karatsuba}.  

Applying Lemma~\ref{Koksma-Hlawka}, we get
$$
	\biggl| \frac{1}{|K_i (\alpha)|} \sum_{k_i \in K_i(\alpha)} f_{\rho }\left( \frac{\alpha k_i}{2^{\rho}} \right) - \int_0^1 f_{\rho} (x_i) dx_i \biggr| \ll 2^{m + \rho} D_{B, K_i} \left( \frac{\alpha k_i}{2^{\rho}} \right).
$$ Next, by Lemma~\ref{Erdos-Turan-Koksma},
\begin{equation} \label{after-er-tur-kok}
	D_{B, K_i} \left( \frac{\alpha k_i}{2^{\rho}} \right) \les \frac{1}{H_i} + \frac{1}{|K_i (\alpha)|} \sum_{0 < |h_i| \les H_i} \frac{1}{|h_i|} \biggl| \sum_{k_i \in K_i (\alpha)} e\left( h_i \frac{\alpha k_i}{2^{\rho}} \right) \biggr|. 
\end{equation} The contribution to $|S_0|^{2^{m+2}}$ from the first term on the right side of the last inequality can be estimated as
\begin{equation} \label{for_F_i_0}
	F_{i,0}^{2^{m+2}} := (DN)^{2^{m+2}} \frac{2^{m + \rho}}{H_i}.
\end{equation} To evaluate the contribution from the second term
we first separate $\alpha$'s for which $|K_i (\alpha)|$ is small. In this case we apply a trivial bound to the sum over $k_i \in K_i (\alpha)$. Fortunately, for most $D \les \alpha < 2D$ the number $|K_i (\alpha)|$ is close to its expected size $B / 2^{2\mu + 2\sigma}$. By definition of $K_i (\alpha)$ and Lemma~\ref{Vinogradov_cup},
\begin{multline*}
	|K_i (\alpha)| = \sum_{k_i \les B} \mathbbm{1} \left( \left\{  \frac{\alpha k_i}{2^{\lambda_i}} \right\} < \frac{1}{2^{\mu + \sigma}} \right) \mathbbm{1} \left( \left\{  \frac{\gamma \alpha k_i}{2^{\lambda_i}} \right\} < \frac{1}{2^{\mu + \sigma}} \right) = \\
	\sum_{k_i \les B} \biggl( \frac{1}{2^{\mu + \sigma}} + \frac{\Theta}{\tilde H_i} + \sum_{0 < |h_1| \les \tilde H_i} \frac{\Theta_{h_1}}{|h_1|} e\biggl( h_1 \frac{\alpha k_i}{2^{\lambda_i}} \biggr) \biggr) \biggl( 	\frac{1}{2^{\mu + \sigma}} + \frac{\Theta}{\tilde H_i} + \sum_{0 < |h_2| \les \tilde H_i} \frac{\Theta_{h_2}}{|h_2|} e\biggl( h_2 \frac{\gamma \alpha k_i}{2^{\lambda_i}} \biggr) \biggr),
\end{multline*} where $\lambda_i := \lambda - (i-1)\mu$. Hence,
\begin{equation} \label{size_K_i}
	|K_i (\alpha)| = \frac{B}{2^{2\mu + 2\sigma}} + \frac{2\Theta B}{2^{\mu + \sigma} \tilde H_i} + \frac{ \Theta^2 B}{\tilde H_i^2} + O\biggl( \frac{B}{\tilde H_i} \sum_{0 < |h| \les \tilde H_i} \frac{1}{|h|} \biggr) + R_1 + R_2 + R_3, 
\end{equation} where
\begin{gather*}
	R_1 := \frac{1}{2^{\mu + \sigma}} \sum_{0 < |h_1| \les \tilde H_i} \frac{\Theta_{h_1}}{|h_1|} \sum_{k_i \les B} e\biggl( h_1 \frac{\alpha k_i}{2^{\lambda_i}} \biggr), \\	R_2 := \frac{1}{2^{\mu + \sigma}} \sum_{0 < |h_2| \les \tilde H_i} \frac{\Theta_{h_2}}{|h_2|} \sum_{k_i \les B} e\biggl( h_2 \frac{\gamma \alpha k_i}{2^{\lambda_i}} \biggr), \\
	R_3 := \sum_{0 < |h_1|, |h_2| \les \tilde H_i} \frac{\Theta_{h_1} \Theta_{h_2}}{|h_1| |h_2|} \sum_{k_i \les B} e\biggl( h_1 \frac{\alpha k_i}{2^{\lambda_i}} + h_2 \frac{\gamma \alpha k_i}{2^{\lambda_i}} \biggr).
\end{gather*} The second, third, and fourth terms on the right side of~\eqref{size_K_i} are negligible if $\tilde H_i \gg 2^{2(\mu + \sigma)(1+\varepsilon)}$ for some $\varepsilon > 0$. So it is enough to show that $R_1, R_2, R_3 \ll B / 2^{2(\mu + \sigma)(1 + o(1))}$ for most~$\alpha$. By Chebyshev's inequality,
\begin{multline*}
	\mathcal{M}_j := \text{meas} \left\{ D \les \alpha < 2D: |R_j (\alpha)| > \frac{B}{2^{2(\mu + \sigma)(1+o(1))}} \right\} \ll \\
	\frac{2^{2(\mu + \sigma)(1 + o(1))}}{B} \int_D^{2D} |R_j (\alpha)| d\alpha
\end{multline*} for $j = 1, 2, 3$. Then, by Lemma~\ref{expsums_estimates},
\begin{multline*}
	\mathcal{M}_1, \mathcal{M}_2 \ll 
	D \frac{2^{(\mu + \sigma)(1+o(1))}}{B} \sum_{0 < |h| \les \tilde H_i} \frac{1}{|h|} \max \left( \log N,  \frac{2^{\lambda_i}}{D |h|} \right) \ll \\
	D \frac{2^{(\mu + \sigma)(1+o(1))}}{B} \frac{2^{\lambda}}{D} \log \tilde H_i,
\end{multline*} and
\begin{multline*}
	\mathcal{M}_3 \ll D \frac{2^{2(\mu + \sigma)(1+o(1))}}{B} \sum_{0 < |h_1|, |h_2| \les \tilde H_i} \frac{1}{|h_1| |h_2|} \max \left( \log N, \frac{2^{\lambda_i}}{D|h_1 + \gamma h_2|} \right) \ll \\
	D \frac{2^{2(\mu + \sigma)(1+o(1))}}{B} \frac{2^{\lambda} N^{1/2}}{D} (\log \tilde H_i)^2.
\end{multline*} These estimates are valid as soon as  $2^{\lambda} \gg DN^{\varepsilon}$, $\tilde H_i \ll N^{\theta_1}$. Then, by the union bound,
\begin{multline} \label{meas_total}
	\mathcal{M} := \text{meas} \left\{ D \les \alpha < 2D : |K_i (\alpha)| < \frac{B}{2^{2(\mu + \sigma)(1 + o(1))}} \right\} \ll \\
	D \frac{2^{2(\mu + \sigma)(1+o(1))}}{B} \frac{2^{\lambda} N^{1/2}}{D} (\log N)^2.
\end{multline} Thus, it is enough to require
$$
	BD \gg N^{1/2} 2^{\lambda + 10 (\mu + \sigma)}, \qquad 2^{3\mu + 3\sigma} \ll \tilde H_i \ll N^{\theta_1}.
$$

Now assuming $\alpha \in \mathcal{M}$ we estimate the sum over $k_i \in K_i (\alpha)$ in~\eqref{after-er-tur-kok} trivially by its length. Then from~\eqref{meas_total} the contribution to $|S_0|^{2^{m+2}}$ from $\alpha \in \mathcal{M}$ is bounded by
\begin{equation} \label{for_F_i_exc}
	F_{i,-1}^{2^{m+2}} :=
	(DN)^{2^{m+2}} 2^{m + \rho} \frac{2^{2(\mu + \sigma)}}{B} \frac{N^{1/2} 2^{\lambda}}{D} (\log N)^3
\end{equation} as soon as $H_i \ll N$. 

Next, we estimate the sum over $k_i \in K_i (\alpha)$ in~\eqref{after-er-tur-kok} for $\alpha \in [D, 2D) \setminus \mathcal{M}$. We rewrite the inner sum as
$$
	\frac{1}{|K_i (\alpha)|} \sum_{k_i \les B} \mathbbm{1} \left( \left\{  \frac{\alpha k_i}{2^{\lambda_i}} \right\} < \frac{1}{2^{\mu + \sigma}} \right) \mathbbm{1} \left( \left\{  \frac{\gamma \alpha k_i}{2^{\lambda_i}} \right\} < \frac{1}{2^{\mu + \sigma}} \right) e\left( h_i \frac{\alpha k_i}{2^{\rho}} \right).
$$ Applying Lemma~\ref{Vinogradov_cup} to each of the factors, we get
\begin{multline*}
	\frac{1}{|K_i (\alpha)|} \biggl| \sum_{k_i \in K_i (\alpha)} e\left( h_i \frac{\alpha k_i}{2^{\rho}} \right) \biggr| \les \\
	\frac{1}{|K_i (\alpha)|} \biggl| \sum_{k_i \les B} e\left( h_i \frac{\alpha k_i}{2^{\rho}} \right) \biggl( \frac{1}{2^{\mu + \sigma}} + \frac{\Theta}{L_i} + \sum_{0 < |l_1| \les L_i} \frac{\Theta_{l_1}}{|l_1|} e\left( l_1 \frac{\alpha k_i}{2^{\lambda_i}} \right) \biggr) \cdot \\
	\biggl( \frac{1}{2^{\mu + \sigma}} + \frac{\Theta}{L_i} + \sum_{0 < |l_2| \les L_i} \frac{\Theta_{l_2}}{|l_2|} e\left( l_2 \frac{\gamma \alpha k_i}{2^{\lambda_i}} \right) \biggr) \biggr| \ll S_1 + S_2 + S_3 + S_4,
\end{multline*} where $\Theta$, $\Theta_{l_1}$, and $\Theta_{l_2}$ are complex numbers bounded by an absolute constant, and
\begin{gather*}
	S_1 := \frac{1}{|K_i (\alpha)|} \frac{1}{2^{2\mu + 2\sigma}} \biggl| \sum_{k_i \les B} e\left( h_i \frac{\alpha k_i}{2^{\rho}} \right) \biggr|, \\
	S_2 := \frac{1}{|K_i (\alpha)|} \sum_{k_i \les B} \frac{1}{L_i} \left( \frac{1}{2^{\mu + \sigma}} + \frac{1}{L_i} + \log L_i \right), \\
	S_3 := \frac{1}{|K_i (\alpha)|} \sum_{0 < |l_1| \les L_i} \frac{1}{|l_1|} \frac{1}{2^{\mu + \sigma}} \biggl| \sum_{k_i \les B} e\left( h_i \frac{\alpha k_i}{2^{\rho}} + l_1 \frac{\alpha k_i}{2^{\lambda_i}} \right) \biggr|,  \\
	S_4 := \frac{1}{|K_i (\alpha)|} \sum_{0 < |l_1|, |l_2| \les L_i} \frac{1}{|l_1| |l_2|}  \biggl| \sum_{k_i \les B} e\left( h_i \frac{\alpha k_i}{2^{\rho}} + l_1 \frac{\alpha k_i}{2^{\lambda_i}} + l_2 \frac{\gamma \alpha k_i}{2^{\lambda_i}} \right) \biggr|.
\end{gather*} The remaining sums can be omitted due to the symmetry. We denote the contributions to $|S_0|^{2^{m+2}}$ from $S_1, S_2, S_3, S_4$ by $F_{i,1}^{2^{m+2}}$, $F_{i,2}^{2^{m+2}}$, $F_{i,3}^{2^{m+2}}$, $F_{i,4}^{2^{m+2}}$ correspondingly. Since $\alpha \notin \mathcal{M}$, we have
$$
	\frac{1}{|K_i (\alpha)|} \les \frac{2^{2(\mu + \sigma)(1+o(1))}}{B}.
$$ Furthermore, the integral over $\alpha \in [D, 2D) \setminus \mathcal{M}$ can be trivially bounded by the integral over  $\alpha \in [D, 2D)$.
 
We start with estimating the contribution to $|S_0|^{2^{m+2}}$ from $S_1$. We have:
$$
	S_1 \ll \frac{2^{o(1)(\mu + \sigma)}}{B} \biggl| \sum_{k_i \les B} e\left( h_i \frac{\alpha k_i}{2^{\rho}} \right) \biggr|.
$$ Then by Lemma~\ref{expsums_estimates} and~\eqref{after-er-tur-kok} the contribution is at most
\begin{equation} \label{for_F_i_1}
	F_{i,1}^{2^{m+2}} :=
	(DN)^{2^{m+2}} \frac{2^{m + \rho + \mu + \sigma} \log H_i}{B} \log B \ll (DN)^{2^{m+2}} \frac{2^{m + \rho + \mu + \sigma}}{B} (\log N)^2.
\end{equation}

Next, for the contribution from $S_2$ we have
$$
	S_2 \ll \frac{1}{|K_i (\alpha)|} \sum_{k_i \les B} \frac{1}{L_i} \log L_i \ll \frac{2^{3\mu + 3\sigma} \log L_i}{L_i},
$$ and so
\begin{multline} \label{for_F_i_2}
	F_{i, 2}^{2^{m+2}} := (DN)^{2^{m+2}} 2^{m + \rho} \log H_i \frac{2^{3\mu + 3\sigma} \log L_i}{L_i} \ll \\
	(DN)^{2^{m+2}} \frac{2^{m + \rho + 3\mu + 3\sigma}}{L_i} (\log N)^2.
\end{multline} This is negligible as soon as $L_i$ is much larger than $2^{\rho + 3\mu + 3\sigma}$ (note that the parameter $m \approx \lambda / \mu$ is of constant size). Because of the error term coming from $S_3$, it is not possible to choose $L_i$ large enough when $i$ is close to $m$. This is the reason why we need to modify our approach to treat the last few sums over $k_i \in K_i (\alpha)$.

The contribution from $S_3$ is estimated similarly to $S_1$ as soon as we require
\begin{equation} \label{approx_phase}
	h_i \frac{\alpha k_i}{2^{\rho}} + l_1 \frac{\alpha k_i}{2^{\lambda_i}} \neq 0,
\end{equation} which is the case for all $|h_i| \les H_i$, $|l_1| \les L_i$ if
$$
	2^{(\lambda_i - \rho)(1 - \varepsilon)} \gg L_i \qquad \text{for} \qquad i \les m-l, \quad \varepsilon > 0.
$$ Recall that $\lambda_i = \lambda - (i-1)\mu$, \ $\rho = \lambda - m\mu$. So it is enough to have $L_i \ll 2^{l\mu}$. Note that from~\eqref{for_F_i_2} we need to require also $L_i \gg 2^{(\rho + 3\mu + 3\sigma)(1+\varepsilon)}$. This means that we should have at least $l \ges 4$. Then we easily get
\begin{equation} \label{for_F_i_3}
	F_{i,3}^{2^{m+2}} := 2^{\mu + \sigma} (\log N) F_{i,1}^{2^{m+2}}.
\end{equation}

Finally, the contribution from $S_4$ by Lemma~\ref{expsums_estimates} does not exceed
\begin{multline} \label{for_F_i_4}
	  N (DN)^{2^{m+2}-1} 2^{m + \rho} \frac{2^{3\mu + 3\sigma}}{B} \sum_{0 < |h_i| \les H_i} \frac{1}{|h_i|} \sum_{0 < |l_1|, |l_2| \les L_i} \frac{1}{|l_1| |l_2|} \cdot \\
	  \max \left( D \log N, \frac{2^{\lambda_i}}{|h_i 2^{\lambda_i - \rho} + l_1 + \gamma l_2|} \right)  \ll (DN)^{2^{m+2}} \frac{2^{ m + \rho + 3\mu + 3\sigma}}{B} (\log N)^4 =: F_{i,4}^{2^{m+2}}.
\end{multline} 

We finish this section by noting that the dual sums over $k_i' \in K_i (\alpha)$ for $i \les m-l$ can be estimated trivially because of 1-periodicity of $g_{\rho}$ and the fact that the discrepancy estimates for the sequences shifted by $-\alpha k_i'$ are the same for each fixed $k_i'$. Thus, we obtain
\begin{multline*}
	|S_0|^{2^{m+2}} \ll N (DN)^{2^{m+2}-1} \int_D^{2D} \max_{\beta \ges 0} \int_{[0, 1]^{m-l+3}} \frac{1}{|K_{m-l+1} (\alpha)|^2 \ldots |K_m (\alpha)|^2} \\
	\sum_{\substack{k_{m-l+1} \in K_{m-l+1} (\alpha) \\ k_{m-l+1}' \in K_{m-l+1} (\alpha)}} \ldots 
	\sum_{k_m, k_m' \in K_m (\alpha)}
	e\biggl( \frac{1}{2}  \sum_{\substack{\varepsilon', \varepsilon_0, \ldots, \varepsilon_m \\ \in \{0, 1\} }} g_{\rho} \bigl(x + \beta + \varepsilon' y + \varepsilon_0 x_0 + \varepsilon_1 x_1 + \ldots + \\ 
	\frac{\alpha \varepsilon_m (k_m - k_m')}{2^{\rho}} \bigr) \biggr) dx_0 dx_1 \ldots dx_{m-l} dx dy d\alpha + \sum_{j=1}^5 E_j^{2^{m+2}} + F_0^{2^{m+2}} + \sum_{i=1}^{m-l} \sum_{j=-1}^4 F_{i,j}^{2^{m+2}}.
\end{multline*}

\section{Proof of Proposition~\ref{prop3}: discrepancy for the remaining sums over $k_i \in K_i (\alpha)$}
\label{Koksma_on_k_last}
	
For values of $i$ close to $m$ (which correspond to small values of $2^{\lambda_i}$) we want to avoid dealing with ``type-$S_3$'' sum from the previous section, which is of the form
$$
	S_3 = \frac{1}{|K_i (\alpha)|} \sum_{0 < |l_1| \les L_i} \frac{1}{|l_1|} \frac{1}{2^{\mu + \sigma}} \biggl| \sum_{k_i \les B} e\left( h_i \frac{\alpha k_i}{2^{\rho}} + l_1 \frac{\alpha k_i}{2^{\lambda_i}} \right) \biggr|.
$$ This is caused by the restrictions on the size of $L_i$ arising from the estimate of ``type-$S_2$'' sum. It forces a large choice of $L_i$, which in the case, for example, $h_i = 1$, $l_1 = -2^{\lambda_i - \rho}$ leads to a sum with zero phase. One way to avoid this issue is to work with the restriction
$$
	\left\{  \frac{\alpha k}{2^{\lambda_i}} \right\} < \frac{1}{2^{\mu + \sigma}}
$$ separately from $\left\{ \alpha k / 2^{\rho} \right\}$. Note that for any integer $n$, the restriction $\{ 2^n x \} \in [a, b]$ implies
$$
	\{ x \} \in \bigcup_{j=0}^{2^n-1} \left[ \frac{a + j}{2^n}, \frac{b + j}{2^n} \right].
$$ In this way we can express the restriction on $\{ \alpha k / 2^{\rho}  \}$ in terms of the restriction on $\{ \alpha k / 2^{\lambda_i} \}$. Precisely, write
$$
	g_{\rho} (x) =: \tilde g_{\rho + (m-i+1)\mu} \left( \frac{x}{2^{\lambda_i - \rho}} \right)
$$ for $(1/2^{\lambda_i - \rho})$-periodic function $\tilde g_{\rho + (m-i+1)\mu}$. Further, let
$$
	\tilde f_{\rho + (m-i+1)\mu} (x) := e\biggl( \frac{1}{2} \sum_{\substack{\varepsilon', \varepsilon_0, \ldots, \varepsilon_m \\ \in \{ 0, 1\}}} \tilde g_{\rho + (m-i+1)\mu} (x) \biggr).
$$ Then
\begin{multline} \label{scaled_sum}
	\frac{1}{|K_i (\alpha)|} \sum_{k_i \in K_i (\alpha)} \tilde f_{\rho + (m-i+1)\mu} \left( \frac{\alpha k_i}{2^{\lambda_i}} \right) = \\ \frac{|\tilde K_i (\alpha)|}{|K_i (\alpha)|} \frac{1}{|\tilde K_i (\alpha)|} \sum_{k_i \in \tilde K_i (\alpha)} \tilde f_{\rho + (m-i+1)\mu} \left( \frac{\alpha k_i}{2^{\lambda_i}}  \right) \mathbbm{1}\left( \left\{ \frac{\alpha k_i}{2^{\lambda_i}} < \frac{1}{2^{\mu + \sigma}} \right\} \right),
\end{multline} where $\tilde K_i (\alpha)$ is the set of $k_i \les B$ with only one restriction $\{ \gamma \alpha k_i / 2^{\lambda_i} \} < 2^{-\mu -\sigma}$. Similarly to Section~\ref{Koksma_on_k_first} we can show that for most $\alpha \in [D, 2D)$ one has
$$
	|K_i (\alpha)| \approx \frac{B}{2^{2\mu + 2\sigma}}, \qquad
	|\tilde K_i (\alpha)| \approx \frac{B}{2^{\mu + \sigma}} 
$$ at the same time. By the union bound we can treat both of these conditions separately. Moreover, since now $i$ is large, the numbers $\lambda_i$ are small, so we have
$$
	\max \left(\log N, \frac{2^{\lambda_i}}{D|h_1|} \right) \ll \log N, \qquad \max \left(\log N, \frac{2^{\lambda_i}}{D|h_1 + \gamma h_2|} \right) \ll \log N.
$$ By repeating the steps from the preivous section, we bound the size of the exceptional set
\begin{multline*}
	\mathcal{M} := \text{meas} \biggl\{ D \les \alpha < 2D : |K_i (\alpha)| < \frac{B}{2^{2(\mu + \sigma)(1+o(1))}} \text{ \ or \ } \\
	|\tilde K_i (\alpha)| < \frac{B}{2^{(\mu + \sigma)(1+o(1))}} \text{ \ or \ } |\tilde K_i (\alpha)| > \frac{B}{2^{(\mu + \sigma)(1-o(1))}} \biggr\} 
\end{multline*} as
$$
	\mathcal{M} \ll D \frac{2^{2(\mu + \sigma)(1+o(1))}}{B} (\log N)^4.
$$ Evaluating the sum on the left side of~\eqref{scaled_sum} trivially for $\alpha \in \mathcal{M}$ we get for the contribution to $|S_0|^{2^{m+2}}$:
\begin{equation} \label{for_G_i_exc}
	G_{i,-1}^{2^{m+2}} := (DN)^{2m+2} \frac{2^{3\mu + 3\sigma}}{B} (\log N)^4.
\end{equation}

Not assume $\alpha \in [D, 2D) \setminus \mathcal{M}$. Applying Lemma~\ref{Koksma-Hlawka} to~\eqref{scaled_sum} we obtain
\begin{multline} \label{after_trick}
	2^{(\mu + \sigma)(1+o(1))} \biggl| \frac{1}{|\tilde K_i (\alpha)|} \sum_{k_i \in \tilde K_i (\alpha)} \tilde f_{\rho + (m-i+1)\mu} \left( \frac{\alpha k_i}{2^{\lambda_i}}  \right) \mathbbm{1}\left( \left\{ \frac{\alpha k_i}{2^{\lambda_i}} < \frac{1}{2^{\mu + \sigma}} \right\} \right) - \\ 
	\int_0^{1/2^{(\mu + \sigma)(1+o(1))}} \tilde f_{\rho + (m-i+1)\mu} (u) du \biggr| \les 2^{(\mu + \sigma)(1+o(1))} V_{HK}^{(1)} (\tilde f) D_{B, \tilde K_i} \left( \frac{\alpha k_i}{2^{\lambda_i}} \right) \ll \\
	2^{(\mu + \sigma)(1+o(1)) + m + \rho + l\mu} D_{B, \tilde K_i} \left( \frac{\alpha k_i}{2^{\lambda_i}} \right).
\end{multline}

The discrepancy $D_{B, \tilde K_i}$ is estimated similarly to the previous section. We denote the new error terms as $G_{i,j}$. By Lemma~\ref{Erdos-Turan-Koksma},
$$
	D_{B, \tilde K_i} \left( \frac{\alpha k_i}{2^{\lambda_i}} \right) \les \frac{1}{H_i} + \frac{1}{|\tilde K_i (\alpha)|} \sum_{0 < |h_i| \les H_i} \frac{1}{|h_i|} \biggl| \sum_{k_i \in \tilde K_i (\alpha)} e\left( h_i \frac{\alpha k_i}{2^{\lambda_i}} \right) \biggr|. 
$$ The first term on the right side of this inequality contributes to $|S_0|^{2^{m+2}}$ at most
\begin{equation} \label{for_G_i_0}
	G_{i,0}^{2^{m+2}} := (DN)^{2^{m+2}} \frac{2^{m + \rho + 2\sigma + (l+2) \mu}}{H_i}.
\end{equation} For the second term by Lemma~\ref{Vinogradov_cup} we have 
$$
	\frac{1}{|\tilde K_i (\alpha)|} \sum_{k_i \les B} e\left( h_i \frac{\alpha k_i}{2^{\lambda_i}}  \right) \biggl( \frac{1}{2^{\mu + \sigma}} + \frac{\Theta_{k_i}}{L_i} + \sum_{0 < |l_i| \les L_i} \frac{1}{|l_i|} e\left( l_i \frac{\gamma \alpha k_i}{2^{\lambda_i}} \right) \biggl) \ll \tilde S_1 + \tilde S_2 + \tilde S_4,
$$ where
\begin{gather*}
	\tilde S_1 := \frac{1}{|\tilde K_i (\alpha)|} \frac{1}{2^{\mu + \sigma}} \biggl| \sum_{k_i \les B} e\left( h_i \frac{\alpha k_i}{2^{\lambda_i}} \right) \biggr|, \qquad
	\tilde S_2 := \frac{1}{| \tilde K_i (\alpha)|} \sum_{k_i \les B} \frac{1}{L_i}, \\
	\tilde S_4 := \frac{1}{|\tilde K_i (\alpha)|} \sum_{0 < |l_i| \les L_i} \frac{1}{|l_i|} \biggl| \sum_{k_i \les B} e\left( h_i \frac{\alpha k_i}{2^{\lambda_i}} + l_i \frac{\gamma \alpha k_i}{2^{\lambda_i}} \right) \biggr|.
\end{gather*} Note that we do not have a ``type-$S_3$'' sum any more. The contributions from $\tilde S_1, \tilde S_2, \tilde S_4$ to $|S_0|^{2^{m+2}}$ can be evaluated similarly to $F_{i,1}^{2^{m+2}}$, $F_{i,2}^{2^{m+2}}$, $F_{i,4}^{2^{m+2}}$. One can easily verify the estimates  
\begin{gather} \label{for_G_i_1}
	G_{i,1}^{2^{m+2}} := (DN)^{2^{m+2}} \frac{2^{m + \rho + 3l(\mu + \sigma)}}{B} (\log N)^2,  \\ \label{for_G_i_2}
	G_{i,2}^{2^{m+2}} := (DN)^{2^{m+2}} \frac{2^{m + \rho + 3l(\mu + \sigma)}}{L_i} (\log N)^2, \\
	\label{for_G_i_4}
	G_{i,4}^{2^{m+2}} := (DN)^{2^{m+2}} \frac{2^{m + \rho + 3l(\mu + \sigma)}}{B} (\log N)^4.
\end{gather} Now we can choose $L_i$ to be very large compared to $2^{m + \rho + 3l(\mu+\sigma)}$.

Next, from~\eqref{after_trick}, the main term in the current bound for $|S_0|^{2^{m+2}}$ is of the form
\begin{multline*}
	N (DN)^{2^{m+2}-1} 2^{l(\mu + \sigma)} \int_D^{2D} \max_{\beta \ges 0} \int_{[0,1]^{m-l+3}} \int_{\left[ 0, \frac{1}{2^{\mu+\sigma}}\right]^l } e\biggl( \frac{1}{2} \sum_{\substack{\varepsilon', \varepsilon_0, \ldots, \varepsilon_m \\ \in \{ 0, 1\}}} g_{\rho} \bigl( x + \beta + \varepsilon'y + \varepsilon_0 x_0 + \\ 
	\varepsilon_1 x_1 + \ldots + \varepsilon_{m-l} x_{m-l} + \varepsilon_{m-l+1} 2^{\lambda_{m-l+1} - \rho} x_{m-l+1} + \ldots + \varepsilon_m 2^{\lambda_m - \rho} x_m \bigr) \biggr)
	dx dy dx_0 dx_1 \ldots dx_m.
\end{multline*} Changing the variables $y_i = 2^{\lambda_i - \rho} x_i$ for $i > m-l+1$, we write the last expression as
\begin{multline*}
	N (DN)^{2^{m+2}-1} \int_D^{2D} \max_{\beta \ges 0} \int_{[0, 1]^{m-l+3}} \biggl( \prod_{i=m-l+1}^m \frac{1}{2^{\lambda_i - \rho - \mu - \sigma}} \int_0^{2^{\lambda_i - \rho - \mu - \sigma}} \biggr) \\
	e\biggl( \frac{1}{2} \sum_{\substack{\varepsilon', \varepsilon_0, \ldots, \varepsilon_m \\ \in \{ 0, 1\}}} 
	g_{\rho} \bigl( x + \beta + \varepsilon'y + \varepsilon_0 x_0 + \ldots + \varepsilon_{m-l} x_{m-l} + \varepsilon_{m-l+1} y_{m-l+1} + \ldots + \varepsilon_m y_m \bigr) \biggr) \\
	dx dy dx_0 dx_1 \ldots dx_{m-l} dy_{m-l+1} \ldots dy_m.	 
\end{multline*} Note that for $i < m$ one clearly has $\lambda_i - \rho - \mu - \sigma > 0$ as soon as $\sigma < \mu$. This means that for $i < m$ we have
$$
	\frac{1}{2^{\lambda_i - \rho - \mu - \sigma}} \int_0^{2^{\lambda_i - \rho - \mu - \sigma}} e\bigl( \ldots \bigr) dy_i = \int_0^1 e\bigl( \ldots \bigr) dy_i.
$$ To treat the last integral corresponding to $i = m$ we apply Cauchy inequality:
\begin{multline*}
	|S_0|^{2^{m+3}} \ll N^2 (DN)^{2^{m+3}-2} \biggl( \int_D^{2D} \int_{[0, 1]^{m+1}} 2^{\sigma} \int_0^{2^{-\sigma}} 1 dy dx_0 dx_1 \ldots dy_m \biggr) \cdot \\
	\biggl( \int_D^{2D} \max_{\beta \ges 0} \int_{[0, 1]^{m+1}} 2^{\sigma} \int_0^{2^{-\sigma}} \biggl| \int_0^1 e\bigl( \ldots \bigr) dx \biggr|^2 dy dx_0 dx_1 \ldots dy_m \biggr) + \\ 
	\sum_{j=1}^5 E_j^{2^{m+3}} + F_0^{2^{m+3}} + \sum_{i=1}^{m-l} \sum_{j=-1}^4 F_{i,j}^{2^{m+3}} + \sum_{i = m-l+1}^m \sum_{\substack{j = -1 \\ j \neq 3}}^4 G_{i,j}^{2^{m+3}}.
\end{multline*} The integral over $[0, 2^{-\sigma}]$ in the second factor of the main term can be trivially bounded by the integral over $[0, 1]$. This gives
\begin{multline*}
	|S_0|^{2^{m+3}} \ll (DN)^{2^{m+3}} 2^{\sigma} \int_{[0,1]^{m+4}} e \biggl( \frac{1}{2} \sum_{\substack{\varepsilon', \varepsilon'', \varepsilon_0, \ldots, \varepsilon_m \\ \in \{ 0, 1\}}}  g_{\rho} \bigl( x + \varepsilon'y + \varepsilon''z + \varepsilon_0 x_0 + \ldots \\
	\ldots + \varepsilon_m y_m \bigr) \biggr) dx dy dz dx_0 \ldots dy_m + \sum_{j=1}^5 E_j^{2^{m+3}} + F_0^{2^{m+3}} + \sum_{i=1}^{m-l} \sum_{j=-1}^4 F_{i,j}^{2^{m+3}} + \\
	\sum_{i = m-l+1}^m \sum_{\substack{j = -1 \\ j \neq 3}}^4 G_{i,j}^{2^{m+3}}.
\end{multline*} The shift $+\beta$ can be removed since there is no dependency on $\alpha$ in the argument of $g_{\rho}$ any more.

\section{Proof of Proposition~\ref{prop3}: Gowers norm estimate}
\label{pf_Gowers_norm}

From the last inequality of Section~\ref{Koksma_on_k_last} we conclude
\begin{multline} \label{all_together}
	|S_0|^{2^{m+3}} \ll (DN)^{2^{m+3}} 2^{\sigma} I_{m+4} (g_{\rho}) + \\
	\sum_{j=1}^5 E_j^{2^{m+3}} + F_0^{2^{m+3}} + \sum_{i=1}^{m-l} \sum_{j=-1}^4 F_{i,j}^{2^{m+3}} +
	\sum_{i = m-l+1}^m \sum_{\substack{j = -1 \\ j \neq 3}}^4 G_{i,j}^{2^{m+3}},
\end{multline} where $I_{m+4}$ is the Gowers $(m+4)$-fold integral 
$$
	I_{m+4} (g_{\rho}) = \int_{[0, 1]^{m+4}} e\biggl( \frac{1}{2} \sum_{\bm{\varepsilon} \in \{0, 1\}^{m+4}} g_{\rho} \bigl( x + \langle \bm{\varepsilon}, \bm{y} \rangle \bigr) dx d\bm{y} \biggr).
$$ Alternatively, it can be seen as the continuous version of the Gowers $(m+4)$-uniformity norm of $e(g_{\rho}/2)$. 
The discrete version of the Gowers norm for this function is defined as
$$
	\norm{e(g_{\rho}/2)}_{U^m (\mathbb{T})}^{2^m} := \frac{1}{2^{(m+1) \rho}} \sum_{\substack{0 \les n < 2^{\rho} \\ 0 \les r_1, \ldots, r_m < 2^{\rho}}} \prod_{\bm{\varepsilon} \in \{0, 1\}^m } e \biggl( \frac{1}{2} g_{\rho} \bigl( \frac{n}{2^{\rho}} + \frac{\langle \bm{\varepsilon}, \bm{r} \rangle}{2^{\rho}}  \bigr)  \biggr).
$$ It was shown by Spiegelhofer~\cite[Proposition~3.3]{Spieg} (after Konieczny~\cite{Kon19}) that this norm decays rapidly with the number of digits $\rho \to +\infty$:

\begin{lemma} \label{Koniec}
	Let $m \ges 2$ be an integer. There exist some constants $\eta_1 > 0$ and $C$ $=$ $C(m)$ $>$ $0$ such that
	$$
		\norm{e(g_{\rho}/2)}_{U^m (\mathbb{T})}^{2^m} \les C 2^{-\eta_1 \rho }
	$$ for any $\rho \ges 0$. 
\end{lemma}

We need an analogue of this lemma for the continuous case. We show that the integral version of the Gowers norm of $g_{\rho} / 2$ does not exceed the discrete one:

\begin{lemma} \label{Gowers_norm_estimate}
	Let $m \ges 2$ be an integer. Then
	$$
		\int_{[0, 1]} \int_{[0, 1]^m} e\biggl( \frac{1}{2} \sum_{\bm{\varepsilon} \in \{0, 1\}^m}  g_{\rho} \left( x + \langle \bm{\varepsilon}, \bm{y} \rangle \right) \biggr) dx d\bm{y} \les \norm{e(g_{\rho}/2)}_{U^m (\mathbb{T})}^{2^m}
	$$ for any $\rho \ges 0$.
\end{lemma}

\begin{proof}
	Let us denote as usual
	$$
		f_{\rho}(x) := e\bigl( \frac{1}{2} g_{\rho} (x) \bigr).
	$$
	
	First, we split the integral to a sum of integrals over small hypercubes in the following way:
	$$
		\int_{[0, 1]} \int_{[0, 1]^m} e\bigl( \ldots \bigr) = \sum_{0 \les n, \bm{r} < 2^{\rho}} \int_{\left[ \frac{n}{2^{\rho}}, \frac{n+1}{2^{\rho}}  \right]} \int \ldots \int_{\left[ \frac{r_i}{2^{\rho}}, \frac{r_{i+1}}{2^{\rho}}  \right]} e\bigl( \ldots \bigr).
	$$ Next, each of $2^{(m+1)\rho}$ small hypercubes can be further partitioned as
	$$
		\left[ \frac{n}{2^{\rho}}, \frac{n+1}{2^{\rho}} \right] \times \left[ \frac{r_1}{2^{\rho}}, \frac{r_1+1}{2^{\rho}} \right] \times \ldots \times \left[ \frac{r_m}{2^{\rho}}, \frac{r_{m+1}}{2^{\rho}} \right] = V_1 \cup \ldots \cup V_R,
	$$ where
	$$
		\sum_{i=1}^R |V_i| = \frac{1}{2^{(m+1)\rho}},
	$$ and on each subset $V_i$ the integrand can be uniquely written as
	$$
		\prod_{\bm{\varepsilon} \in \{0, 1\}^m} f_{\rho} \biggl( \frac{n}{2^{\rho}} + \frac{\langle \bm{\varepsilon}, \bm{r} \rangle + t_{\bm{\varepsilon}}^{(V_i)}}{2^{\rho}} \biggr) \qquad \text{for some integer } t_{\bm{\varepsilon}}^{(V_i)} > 0.
	$$ Then
	\begin{multline*}
		I_m (g_{\rho}) = \sum_{0 \les n, \bm{r} < 2^{\rho}} \sum_{i=1}^R \int_{V_i} \prod_{\bm{\varepsilon} \in \{0, 1\}^m} f_{\rho} \biggl( \frac{n}{2^{\rho}} + \frac{\langle \bm{\varepsilon}, \bm{r} \rangle + t_{\bm{\varepsilon}}^{(V_i)}}{2^{\rho}} \biggr) dx d\bm{y} = \\
		\sum_{i=1}^R |V_i| \sum_{0 \les n, \bm{r} < 2^{\rho}} \prod_{\bm{\varepsilon} \in \{0, 1\}^m} f_{\rho} \biggl( \frac{n}{2^{\rho}} + \frac{\langle \bm{\varepsilon}, \bm{r} \rangle + t_{\bm{\varepsilon}}^{(V_i)}}{2^{\rho}} \biggr).
	\end{multline*} The last expression can be seen as a multidimensional Gowers inner product of $2^m$ functions. Applying the Gowers modification of the Cauchy inequality (see~\cite[Lemma~3.8]{Gow01}), we obtain
	$$
		I_m (g_{\rho}) \les \sum_{i=1}^R |V_i| \prod_{\bm{\varepsilon} \in \{0, 1\}^m} \biggl| \sum_{0 \les n, \bm{r} < 2^{\rho}} \prod_{\bm{\eta} \in \{ 0, 1 \}^m } f_{\rho} \biggl( \frac{n}{2^{\rho}} + \frac{\langle \bm{\eta}, \bm{r} \rangle + t_{\bm{\varepsilon}}^{(V_i)}}{2^{\rho}} \biggr) \biggr|^{1/2^m}.
	$$ Next, notice that since the product over the vectors $\bm{\varepsilon}$ is now outside, we get the discrete Gowers norm of the shifted version of $f_{\rho}$, which is clearly the same as for the original function $f_{\rho}$. This implies
	\begin{multline*}
		I_m (g_{\rho}) \les
		\sum_{i=1}^R |V_i| \prod_{\bm{\varepsilon} \in \{0, 1\}^m} \left( 2^{(m+1)\rho} \right)^{1/2^m} \cdot \\  \biggl| \frac{1}{2^{(m+1)\rho}} \sum_{0 \les n, \bm{r} < 2^{\rho}} \prod_{\bm{\eta} \in \{ 0, 1 \}^m } f_{\rho} \biggl( \frac{n}{2^{\rho}} + \frac{\langle \bm{\eta},  \bm{r} \rangle + t_{\bm{\varepsilon}}^{(V_i)}}{2^{\rho}} \biggr) \biggr|^{1/2^m} =
		\norm{f_{\rho}}_{U^m}^{2^m}
	\end{multline*} as desired.
	
\end{proof}

Applying Lemmas~\ref{Gowers_norm_estimate} and~\ref{Koniec} to the main term in~\eqref{all_together}, we find
\begin{multline} \label{all_together_2}
	|S_0|^{2^{m+3}} \ll (DN)^{2^{m+3}} 2^{\sigma - \eta_1 \rho} + \\
	\sum_{j=1}^5 E_j^{2^{m+3}} + F_0^{2^{m+3}} + \sum_{i=1}^{m-l} \sum_{j=-1}^4 F_{i,j}^{2^{m+3}} +
	\sum_{i = m-l+1}^m \sum_{\substack{j = -1 \\ j \neq 3}}^4 G_{i,j}^{2^{m+3}}.
\end{multline}

\section{Proof of Proposition~\ref{prop3}: adjustment of the parameters}
\label{final_comp}

First, let us recall the bounds for all the error terms from~\eqref{all_together_2}. They are given by formulas~\eqref{for_E_1}, \eqref{for_E_2}, \eqref{for_E_3}, \eqref{for_E_4}, \eqref{for_E_5} (for $E_j$'s); \eqref{for_F_0}, \eqref{for_F_i_0}, \eqref{for_F_i_exc}, \eqref{for_F_i_1}, \eqref{for_F_i_2}, \eqref{for_F_i_3}, \eqref{for_F_i_4} (for $F_{i,j}$'s); \eqref{for_G_i_exc}, \eqref{for_G_i_0}, \eqref{for_G_i_1}, \eqref{for_G_i_2}, \eqref{for_G_i_4} (for $G_{i,j}$'s). Assuming that all of them are $\ll (DN)^{m+3}$, we get the following estimates for $2^{m+3}$-powers: \\

1) for $E_j$-terms:
\begin{gather*}
	E_1^{2^{m+3}} \ll (DN)^{2^{m+3}} \left(\frac{1}{R_0} + \frac{R_0 D}{2^{\lambda}} + \frac{R_0}{N} \right), \qquad
	E_2^{2^{m+3}} \ll (DN)^{2^{m+3}} \frac{B}{N}, \\
	E_3^{2^{m+3}} \ll 2^m (DN)^{2^{m+3}} \left( \frac{1}{2^{\sigma-2}} + \frac{(\log N)^2}{N} \right), \qquad
	E_4^{2^{m+3}} \ll (DN)^{2^{m+3}} \frac{2^{m + \rho}}{H}, \\
	E_5^{2^{m+3}} \ll (DN)^{2^{m+3}} \frac{2^{m + \rho} (\log N)^3}{N};
\end{gather*} 

2) for $F_{i,j}$-terms:
\begin{gather*}
	F_0^{2^{m+3}} \ll (DN)^{2^{m+3}} \frac{2^{m + \rho}}{R_0} (\log N)^2, \qquad
	F_{i,0}^{2^{m+3}} \ll (DN)^{2^{m+3}} \frac{2^{m + \rho}}{H_i}, \\
	F_{i,1}^{2^{m+3}} \ll (DN)^{2^{m+3}} \frac{2^{m + \rho + \mu + \sigma}}{B} (\log N)^2, \qquad
	F_{i,2}^{2^{m+3}} \ll (DN)^{2^{m+3}} \frac{2^{m + \rho + 3\mu + 3\sigma}}{L_i} (\log N)^2, \\
	F_{i,3}^{2^{m+3}} \ll (DN)^{2^{m+3}} \frac{2^{m + \rho + 2\mu + 2\sigma}}{B} (\log N)^3, \qquad
	F_{i,4}^{2^{m+3}} \ll (DN)^{2^{m+3}} \frac{2^{3\mu + 3\sigma + m + \rho}}{B} (\log N)^4, \\
	F_{i,-1}^{2^{m+2}} \ll
	(DN)^{2^{m+2}} 2^{m + \rho} \frac{2^{2(\mu + \sigma)}}{B} \frac{N^{1/2} 2^{\lambda}}{D} (\log N)^3, \qquad
	1 \les i \les m-l;
\end{gather*} 

3) for $G_{i,j}$-terms:
\begin{gather*}
	G_{i,0}^{2^{m+3}} \ll (DN)^{2^{m+3}} \frac{2^{m + \rho + 2\sigma + (l+2)\mu}}{H_i}, \qquad
	G_{i,1}^{2^{m+3}} \ll (DN)^{2^{m+3}} \frac{2^{m + \rho + 3l(\mu + \sigma)}}{B} (\log N)^2,  \\
	G_{i,2}^{2^{m+3}} \ll (DN)^{2^{m+3}} \frac{2^{m + \rho + 3l(\mu + \sigma)}}{L_i} (\log N)^2, \qquad
	G_{i,4}^{2^{m+3}} \ll (DN)^{2^{m+3}} \frac{2^{m + \rho + 3l(\mu + \sigma)}}{B} (\log N)^4, \\
	G_{i,-1}^{2^{m+2}} \ll (DN)^{2m+2} \frac{2^{3\mu + 3\sigma}}{B} (\log N)^4, \qquad
	m-l+1 \les i \les m.
\end{gather*}

Furthermore, throughout the computation we obtained the following restrictions: $H_i \ll N$ and $H, L_i \ll N^{\theta_1}$ for all $1 \les i \les m$; \ $2^{\rho + 4\mu + 4\sigma} \ll L_i \ll 2^{l\mu}$ for $1 \les i \les m-l$. Note that the implied constants in all the estimates above can depend on $m$, which is roughly of size $\lambda / \mu$. The last fraction will only depend on $c$ in Theorem~\ref{thm1}. The desired bound would follow from the choice of the parameters satisfying the aforementioned conditions together with
\begin{gather*}
	R_0, B, H_i \ll N^{1-\kappa_1} \quad \text{and} \quad 
	R_0, B, H_i, H \gg \max \left( 2^{10 \rho}, 2^{10l(\mu + \sigma)} \right) \quad \text{for } 1 \les i \les m, \\
	\eta_1 \rho \ges 10 \sigma, \qquad \mu \ges 10 \rho, \\
	2^{\sigma}, 2^{\mu} \gg N^{\kappa_2}, \qquad
	R_0 D \ll 2^{\lambda} N^{-\kappa_3}, \qquad
	BD \gg 2^{\lambda + \rho + 2\mu + \sigma} N^{1/2 + \kappa_4}, \\
	\max \left( 2^{10 \rho}, 2^{10l(\mu + \sigma)} \right) \ll L_i \ll N^{1-\kappa_5} \qquad \text{for } m-l < i \les m,
\end{gather*} for some $\kappa_1, \kappa_2, \kappa_3, \kappa_4, \kappa_5 > 0$. All the restrictions are satisfied, for example, with the following choice:
\begin{gather*}
	l = 10, \qquad \mu = \floor{\frac{\theta_1 \log_2 N}{200}}, \qquad \rho = \floor{\frac{\mu}{10}}, \qquad \sigma = \floor{\frac{\eta_1 \rho}{10}}, \\
	\lambda = \floor{\log_2 DN^{1/5}}, \qquad R_0 = N^{1/10}, \qquad H = N^{9\theta_1/ 10}, \\ 
	B = H_i = N^{4/5} \quad \text{for } 1 \les i \les m,  \\
	L_i	= N^{9\theta_1 / 10} \qquad \text{for } m-l < i \les m, \\
	L_i = 2^{5\mu} \qquad \text{for } 1 \les i \les m-l,
\end{gather*} and the corresponding $\kappa_1, \kappa_2, \kappa_3, \kappa_4, \kappa_5, m > 0$. This completes the proof.

\nocite{*}
\bibliographystyle{abbrv}
\bibliography{Sarnak_TM_bib}

\end{document}